\documentclass{article}
\usepackage{amsmath}
\usepackage{amsfonts}
\usepackage{amsthm}
\usepackage{enumerate}

\newtheorem{thm}{Theorem}
\newtheorem{prop}[thm]{Proposition}
\newtheorem{lem}[thm]{Lemma}
\newtheorem{cor}[thm]{Corollary}

\theoremstyle{definition}
\newtheorem{defn}[thm]{Definition}
\newtheorem{rmk}[thm]{Remark}

\newcommand{\sqbinom}[2]{\big[\begin{smallmatrix} #1 \\ #2 \end{smallmatrix}\big]}

\begin{document}

\title{On Higher Partial Derivatives of Implicit Functions and their Combinatorics}

\author{Shaul Zemel}

\maketitle

\section*{Introduction}

In the previous paper \cite{[Z]}, we considered the higher derivatives of an implicit function of one variable. As the binomial Leibniz rule completes the rule for first derivatives of products, and Fa\`{a} di Bruno's formula extends the chain rule to higher orders, there is a formula, initially proved in \cite{[C1]} and then in \cite{[CF]} (the latter corrected in \cite{[Wi]}), which plays the same role for the well-known formula from first-year calculus, expressing the first derivative of an implicit function as minus the quotient of the partial derivatives of the defining two-variable function. For more on the history of such questions, see \cite{[J1]} (and some of the references therein), while other references that are relevant to the implicit function case are \cite{[Wo]}, \cite{[C2]}, \cite{[S]}, \cite{[N]}, and \cite{[J3]}. Note that the formula from \cite{[Wi]} and \cite{[J3]} contains many terms, and \cite{[Z]} uses certain combinations of derivatives as the building blocks for the formula, yielding another formula, which allows one to write the $n$th derivative of an implicit function using significantly less terms. The special case where the implicit function is an inverse function was considered, together with parametric functions, in \cite{[J2]}.

The formula of Fa\`{a} di Bruno involves ordinary partitions of the order $n$ of the derivative. We recall that $\lambda$ is a partition of $n$, a statement that we denote by $\lambda \vdash n$, meaning that $\lambda$ consists of a decreasing sequence of positive integers $a_{q}$, $1 \leq q \leq p$ that sum to $n$. A more useful description of a partition is by using, for every $j$, the multiplicity $m_{j}$ of $j$ in $\lambda$, counting how many times $j$ shows up in the partition (i.e., how many of the $a_{q}$'s equal $j$). The formulae from \cite{[J2]}, \cite{[J3]}, \cite{[Wi]}, and \cite{[Z]} (among others) are based on partitions with two parameters, i.e., of certain integral vectors, with the coefficient corresponding to any such partition being expressed again in terms of the multiplicities describing the partition.

\medskip

One may ask similar questions for functions of several variables. A systematic construction of the generalization of Fa\`{a} di Bruno's formula to the case where both functions used in the composition have several variables is given in \cite{[CS]}. Both this question and the binomial Leibniz rule were considered combinatorially in \cite{[H]} (see also some of the references cited there, as well as the concise paper \cite{[M]}), where it is also realized that higher derivatives with respect to distinct variables give only coefficients of 1, and the coefficients showing up in other derivatives come only from different terms from the distinct variables case coinciding (this is called ``collapsing partitions'' in that reference). In this paper, we treat the question of an implicit function of several variables.

\smallskip

Now, when $y=y(x)$ is the function defined implicitly by $f(x,y)=0$, there are two variables that can become vector-valued: Either $x$, or $y$ and $f$. The paper \cite{[Y]} already considered the case where both variables are vector-valued, yielding formulae for the derivatives (i.e., the expansion) in terms of various types of integrals, but without any combinatorial meaning. The paper \cite{[A]} established some properties of the expansion itself when $y$ and $f$ are vectors, in terms of a recursive relation, showing how quickly the details become complicated in general (these formulae are then used there for studying the behavior of solutions of rate-distortion questions in some situations). The main issue is the repeated multiplication by the inverse of the Jacobian of $f$ with respect to $y$, and the derivatives of its entries, which is why \cite{[Y]} assumes in many formula that this derivative is the identity matrix up to higher order terms.

So in this paper we keep $f$ and $y$ as scalar-valued functions, let $x$ be a vector-valued variable $\vec{x}=(x_{1},\ldots,x_{N})$, and evaluate all the derivatives of $y$, of any order. The argument follows the one from \cite{[Z]} very closely, and in particular our main formula, Theorem \ref{coeffval}, is given in terms of the larger building blocks, analogously to those from that reference (though we also provide the formula in terms of products of the usual derivatives in Theorem \ref{claselts} later). Note that here also in the case of derivatives with respect to distinct variables, the coefficients do not always equal 1, and when some variables in the derivative coincide, we multiply the basic coefficients from the formula for distinct variables by the number of collapsing partitions from \cite{[H]}. Note that \cite{[H]} begins with the case of distinct variables, and then explains how counting collapsing partitions yields the general case, but for our formulae the argument for distinct variables is not simpler to work with than the general one (especially in terms of notation), so we work with a general derivative to begin with. See Remark \ref{Hardy} at the end for the precise details of this observation, for the formulae from both Theorems \ref{coeffval} and \ref{claselts}.

As usual when considering variables that can either coincide or be distinct, the appropriate language to phrase and prove our formulae is in terms of multi-indices, or equivalently multisets, which are based on the indices of our variables (between 1 and $N$ for $\vec{x}=(x_{1},\ldots,x_{N})$). Each such multiset is determined by the multiplicities with which it contains the indices $1 \leq i \leq N$, and the operation of sums of multisets (corresponding to union of disjoint sets) is expressed in terms of addition of these multiplicities. Indeed, the partitions that we get in our formulae are of vectors of length 2 with the first entry being a multiset (this is equivalent to partitioning vectors of length $N+1$, where the last coordinate plays a role that differs from that of all the others). Several expressions that are based on these multisets (like factorials or binomial coefficients) are given in terms of products of the ordinary expressions, based on the multipicities---we shall define each one of these as it shows up.

\medskip

This paper is divided into 4 sections. Section \ref{Blocks} introduces the building blocks that are used in our main formula. Section \ref{FirstDesc} determines which products of these combinations show up in said formula, while Section \ref{DetCoeff} establishes the value of the combinatorial coefficient with which every such expression shows up. Finally, Section \ref{Expansion} explains the origin of these combinations, and proves the formula for the derivatives using only ordinary products of derivatives.

I would like to thank S. Agmon for stimulating discussions on this subject, as well as on the relevant parts of \cite{[A]}.

\section{The Basic Building Blocks \label{Blocks}}

We begin by fixing some notation for the rest of the paper. As mentioned in the Introduction, we will have $N$ free variables $x_{i}$, $1 \leq i \leq N$, and we shall use $\vec{x}$ for $(x_{1},\ldots,x_{N})$ for short. Given a function $g$ of $\vec{x}$, and perhaps of another variable $y$, we shall write, as in \cite{[Z]} and others, the partial derivative of $g$ with respect to $y$ as simply $g_{y}$, and we shall shorten the partial derivative with respect to the variable $x_{i}$ even further to $g_{i}$. For higher order derivatives of small order we shall simply put one index for each variable with respect to which we differentiate.

In order to discuss higher order derivatives in a more natural way, we introduce, following \cite{[Z]}, the following notation. The derivative $\frac{\partial^{r}g}{\partial y^{r}}$ for some $r\geq0$ will be written as $g_{y^{r}}$. On the other hand, when $I$ is a subset of the numbers between 1 and $N$, whose cardinality we denote by $|I|$, the notation $g_{I}$ will stand for the derivative $\partial^{|I|}g\big/\prod_{i \in I}\partial x_{i}$. More generally, we allow $I$ to be a multiset, in which indices $i$ between 1 and $N$ may appear with multiplicities, and $|I|$ is the sum of these multiplicities. Then we still write $g_{I}$ for $\partial^{|I|}g\big/\prod_{i \in I}\partial x_{i}$, where in the product the number of times that we differentiate with respect to the variable $x_{i}$ is its multiplicity in $I$. In more conventional notation, if $\nu_{i}$ is the multiplicity with which $i$ shows up in $I$, then $|I|=\sum_{i=1}^{N}\nu_{i}$, and $g_{I}$ is $\partial^{|I|}g\big/\prod_{i=1}^{N}\partial x_{i}^{\nu_{i}}$. For a general mixed derivative $\partial^{|I|+r}g\big/\prod_{i=1}^{N}\partial x_{i}^{\nu_{i}} \cdot \partial y^{r}$ we shall use the notation $g_{Iy^{r}}$. We shall view all sets as multisets (meaning that a set is a multiset in which the only multiplicities are 0 and 1), and in particular the empty set is the multiset in which all the indices come with multiplicity 0.

\smallskip

So let $f=f(\vec{x},y)$ be a function of $N+1$ variables, and assume that we are given points $(\vec{x}_{0},y_{0})$, with $\vec{x}_{0}$ standing for $(x_{1,0},\ldots,x_{N,0})$, such that $f$ is continuously differentiable at the point $(\vec{x}_{0},y_{0})$ and satisfies $f(\vec{x}_{0},y_{0})=0$ and $f_{y}(\vec{x}_{0},y_{0})\neq0$. Then the equality $f(\vec{x},y)=0$ defines, by the Implicit Function Theorem, a function $y=y(\vec{x})$ at a neighborhood of $\vec{x}_{0}$ (with $y(\vec{x}_{0})=y_{0}$) that has the same differentiability properties of $f$. All the derivatives that we shall henceforth take will be at either the point $\vec{x}_{0}$ or the point $(\vec{x}_{0},y_{0})$, and we shall not mention these explicit values in the notation.

Now, for the first derivative $y_{i}$, we can keep the variables $x_{j}$, $j \neq i$ as fixed, and use the result from any basic calculus course to deduce that $y_{i}=-\frac{f_{i}}{f_{y}}$. We are then assuming that $f$ is continuously differentiable of order at least $n$ for some $n \geq 2$, and we are interested in an explicit expression, using combinatorial coefficients, for the derivative $y_{I}$ for any multiset $I$ with $|I|=n$ (the latter equality means that the multiset $I$ has \emph{size} $n$), in terms of the partial derivatives of $f$. We shall follow the arguments and examples from \cite{[Z]}.

\smallskip

Let $g$ be a function of $\vec{x}$ and $y$, and consider the function taking $\vec{x}$ to $g\big(\vec{x},y(\vec{x})\big)$. The chain rule implies that its derivative with respect to a variable $x_{j}$ is $g_{j}+g_{y}y_{j}=g_{j}-\frac{g_{y}f_{i}}{f_{y}}$. We therefore obtain
\begin{equation}
\frac{\partial}{\partial x_{j}}f_{i}=\frac{f_{ij}f_{y}-f_{yi}f_{j}}{f_{y}}\qquad\mathrm{and}\qquad\frac{\partial}{\partial x_{j}}f_{y}=\frac{f_{jy}f_{y}-f_{yy}f_{j}}{f_{y}}, \label{ddj1stder}
\end{equation}
from which we can evaluate the second derivative $y_{ij}:=-\frac{\partial}{\partial x_{j}}\frac{f_{i}}{f_{y}}$. One way is through the derivative of a quotient, via Equation \eqref{ddj1stder}, and the other one is by taking $g=\frac{f_{i}}{f_{y}}$ and considering $-\big(\frac{f_{i}}{f_{y}}\big)_{j}+\big(\frac{f_{i}}{f_{y}}\big)_{y}\frac{f_{j}}{f_{y}}$, yielding that $y_{ij}$ equals
\begin{equation}
\frac{-f_{ij}f_{y}+f_{yj}f_{i}}{f_{y}^{2}}+\frac{f_{iy}f_{y}-f_{yy}f_{i}}{f_{y}^{2}}\cdot\frac{f_{j}}{f_{y}}=\frac{-f_{ij}f_{y}^{2}+f_{iy}f_{j}f_{y}+f_{jy}f_{i}f_{y}-f_{yy}f_{i}f_{j}}{f_{y}^{3}}. \label{yder2}
\end{equation}
In Equation \eqref{yder2} we already used the equality of the mixed derivatives when we replaced $f_{yj}$ by $f_{jy}$, and we shall always push the derivatives with respect to $y$ to the end (since $f$ will always be assumed to have enough differentiability properties to justify such operations), in view of the notation $g_{Iy^{r}}$ introduced above. Note that when $i=j$ the two sums with the $+$ sign coincide, producing a factor of 2 as in Equation (2) of \cite{[Z]}.

As in \cite{[J2]} and \cite{[Z]}, we will now prove that the denominator of the derivative $y_{I}$, when $n=|I|$, is $f_{y}^{2n-1}$, via the following lemma, in which we shall write $I+j$ for the multiset sum of $I$ and $\{j\}$ (this is a union if $j \notin I$, and means increasing the multiplicity of $j$ by one while leaving the other multiplicities invariant in general), and note that $y_{I+j}$ is the same as $(y_{I})_{j}$ (as well as $(y_{j})_{I}$).
\begin{lem}
Let $I$ be a multiset of size $n$, and take $1 \leq j \leq N$. Then $f_{y}^{2n+1}y_{I+j}$ can be evaluated as \[f_{y}^{2}\tfrac{\partial}{\partial x_{j}}\big(f_{y}^{2n-1}y_{I}\big)-(2n-1)(f_{jy}f_{y}-f_{yy}f_{j})\big(f_{y}^{2n-1}y_{I}\big).\] \label{denfree}
\end{lem}

\begin{proof}
We differentiate $f_{y}^{2n-1}\big(\vec{x},y(\vec{x})\big)y_{I}(\vec{x})$ with respect to $x_{j}$ via the chain and product rules, and use Equation \eqref{ddj1stder} for evaluating $f_{y}\cdot\frac{\partial}{\partial x_{j}}f_{y}$ (after we multiplied the previous expression by $f_{y}^{2}$). Moving the expression involving $2n-1$ to the other side yields the desired equality. This proves the lemma.
\end{proof}

Let us see how Lemma \ref{denfree} evaluates $f_{y}^{5}y_{ijk}$, using the fact that $f_{y}^{3}y_{ij}$ equals $-f_{ij}f_{y}^{2}+f_{iy}f_{j}f_{y}+f_{jy}f_{i}f_{y}-f_{yy}f_{i}f_{j}$ by Equation \eqref{yder2}. We expand $\frac{\partial}{\partial x_{k}}f_{i}$, $\frac{\partial}{\partial x_{k}}f_{j}$, and $\frac{\partial}{\partial x_{k}}f_{y}$ via Equation \eqref{ddj1stder}, and while in the case of ordinary derivatives from \cite{[Z]}, the resulting parts of $\frac{d}{dx}(f_{y}^{3}y'')$ all canceled, here, with $\frac{\partial}{\partial x_{k}}(f_{y}^{3}y_{ij})$, only those with $f_{iy}f_{yy}f_{j}f_{k}$ and $f_{jy}f_{yy}f_{i}f_{k}$ do in general (i.e., when $i$, $j$, and $k$ are distinct). When we consider the remaining parts and recall the second term from Lemma \ref{denfree}, we deduce that $f_{y}^{5}y_{ijk}$ is the sum of three expressions, one is
\[-f_{ijk}f_{y}^{4}+f_{ijy}f_{k}f_{y}^{3}+f_{iky}f_{j}f_{y}^{3}+f_{jky}f_{i}f_{y}^{3}+\]
\begin{equation}
-f_{iyy}f_{j}f_{k}f_{y}^{2}-f_{jyy}f_{iyy}f_{j}f_{k}f_{y}^{2}-f_{kyy}f_{iyy}f_{i}f_{j}f_{y}^{2}+f_{yyy}f_{i}f_{j}f_{k}f_{y}, \label{yder3D3}
\end{equation}
the other is
\[-2f_{ij}f_{ky}f_{y}^{3}+f_{ik}f_{jy}f_{y}^{3}+f_{jk}f_{iy}f_{y}^{3}+2f_{ij}f_{yy}f_{k}f_{y}^{2}-f_{ik}f_{yy}f_{j}f_{y}^{2}-f_{jk}f_{yy}f_{i}f_{y}^{2}+\]
\begin{equation}
-2f_{iy}f_{jy}f_{k}f_{y}^{2}+f_{iy}f_{ky}f_{j}f_{y}^{2}+f_{jy}f_{ky}f_{i}f_{y}^{2}, \label{killsym}
\end{equation}
and the third one is
\begin{equation}
3(f_{ky}f_{y}-f_{yy}f_{k})(f_{ij}f_{y}^{2}-f_{iy}f_{j}f_{y}-f_{jy}f_{i}f_{y}+f_{yy}f_{i}f_{j}). \label{yder3rest}
\end{equation}
We see, however, that this presentation (which was based on taking the derivative with respect to $x_{k}$ after the others) does not give an expression that is symmetric with respect to interchanging $k$ with $i$ or with $j$. As we know that $f_{y}^{5}y_{ijk}$ must have this symmetry property, we can write $f_{y}^{5}y_{ijk}$ as the average of $\frac{\partial}{\partial x_{k}}(f_{y}^{3}y_{ij})-3(f_{ky}f_{y}-f_{yy}f_{k})(f_{y}^{3}y_{ij})$, $\frac{\partial}{\partial x_{j}}(f_{y}^{3}y_{ik})-3(f_{jy}f_{y}-f_{yy}f_{j})(f_{y}^{3}y_{ik})$, and $\frac{\partial}{\partial x_{i}}(f_{y}^{3}y_{jk})-3(f_{iy}f_{y}-f_{yy}f_{i})(f_{y}^{3}y_{jk})$. Then the expression from Equation \eqref{yder3D3} remains, the ones from Equation \eqref{killsym} indeed cancel out, and Equation \eqref{yder3rest} is replaced by three terms, without the factor 3. One can, in fact, expand Equation \eqref{yder3rest} can be expanded like in \cite{[N]} and \cite{[J3]} to give \[3f_{ij}f_{ky}f_{y}^{3}-3f_{ij}f_{yy}f_{k}f_{y}^{2}-3f_{iy}f_{ky}f_{j}f_{y}^{2}-3f_{jy}f_{ky}f_{i}f_{y}^{2}+\] \[+3f_{iy}f_{yy}f_{j}f_{k}f_{y}+3f_{jy}f_{yy}f_{i}f_{k}f_{y}+3f_{ky}f_{yy}f_{i}f_{j}f_{y}-3f_{yy}^{2}f_{x}f_{i}f_{j}f_{k},\] and verify that adding it to Equation \eqref{killsym} indeed gives the asserted average of Equation \eqref{yder3rest}, expanded in the same way.

We recall from \cite{[Z]} that our basic building blocks in the case of ordinary derivatives were certain combinatorial sums. Based on our formula for $f_{y}^{3}y_{ij}$ and on the first term in the formula for $f_{y}^{5}y_{ijk}$, we define our building blocks in this setting as follows. As with sets, we say that a multiset $K$ is \emph{contained} in another multiset $J$ when the multiplicity $\kappa_{i}$ with which the index $i$ appears in $K$ is smaller than or equal to the multiplicity $\eta_{i}$ with which it shows up in $J$, for every $1 \leq i \leq N$. When this is the case we write $\binom{J}{K}$ for the product $\prod_{i=1}^{N}\binom{\eta_{i}}{\kappa_{i}}$ (this is just 1 in case $J$, hence also $K$, is a set, thus with multiplicities 0 or 1, and vanishes when $K$ is not contained in $J$).
\begin{defn}
Let $g$ be a smooth enough function $g$ of the variables $\vec{x}$ and $y$, and take a multiset $J$ of indices between 1 and $N$. For these parameters we define the expression $\Delta_{J}g=\sum_{K \subseteq J}(-1)^{|K|}\binom{J}{K}g_{(J \setminus K)y^{|K|}}\cdot\prod_{i \in K}f_{i} \cdot f_{y}^{|J|-|K|}$. When $J$ is $\{i\}$, $\{i,j\}$, $\{i,j,k\}$, $\{i,j,k,l\}$ etc. we will write $\Delta_{i}g$, $\Delta_{ij}g$, $\Delta_{ijk}g$, $\Delta_{ijkl}g$, etc. respectively for the corresponding expression $\Delta_{J}g$ (and the same for multisets like $\{i,i\}$, $\{i,i,j\}$, and $\{i,i,i\}$, exhibiting multiplicities). \label{DeltaJg}
\end{defn}
Note that the product $\prod_{i \in K}f_{i}$ in Definition \ref{DeltaJg} means $\prod_{i=1}^{N}f_{i}^{\kappa_{i}}$, with the multiplicities from above, and that $J \setminus K$ is the multiset difference, in which $i$ shows up with the multiplicity $\eta_{i}-\kappa_{i}$.

Definition \ref{DeltaJg} expresses the formula for $f_{y}^{3}y_{ij}$ given in Equation \eqref{yder2} as just $-\Delta_{ij}f$, and the expression for $f_{y}^{5}y_{ijk}$ obtained after the averaging from the last paragraph becomes
\begin{equation}
-f_{y}\Delta_{ijk}f+\Delta_{k}f_{y}\cdot\Delta_{ij}f+\Delta_{j}f_{y}\cdot\Delta_{ik}f+\Delta_{i}f_{y}\cdot\Delta_{jk}f \label{yder3fin}
\end{equation}
from Equation \eqref{yder3D3} and the averaging of Equation \eqref{yder3rest} (the three symmetric summands in Equation \eqref{yder3fin} correspond to the coefficient 3 in the expression $-f_{y}\Delta_{3}f+3\Delta_{1}f_{y}\cdot\Delta_{2}f$ from \cite{[Z]} for $f_{y}^{5}$ times the third ordinary derivative, and this coefficient indeed shows up when $i=j=k$). Like in \cite{[Z]} we shall follow the convention that the index precedes the $\Delta$-sign, namely $\Delta_{k}f_{y}$ from the last expression will always mean $\Delta_{k}(f_{y})$ and never $(\Delta_{k}f)_{y}$, and the same for every such expression that will appear below. Using Definition \ref{DeltaJg}, the formula for $f_{y}\cdot\frac{\partial}{\partial x_{j}}g\big(\vec{x},y(\vec{x})\big)$ is $\Delta_{j}g$ and the numerators from Equation \eqref{ddj1stder} are $\Delta_{j}f_{i}$ and $\Delta_{j}f_{y}$ respectively, and we note again that these are not the derivatives of $\Delta_{j}f$ itself, since the latter vanishes by the definition of $y=y(\vec{x})$ implicitly via $f\big(\vec{x},y(\vec{x})\big)=0$.

\section{The Expressions Appearing in $f_{y}^{2n-1}y_{I}$ \label{FirstDesc}}

Given an expression for $f_{y}^{2n-1}y_{I}$, where $I$ is a multiset of size $n$, in terms of the construction blocks from Definition \ref{DeltaJg}, we would like to get, for an index $k$, the formula for $f_{y}^{2n+1}y_{I+k}$ using Lemma \ref{denfree} (recall that $I+k$ is our shorthand for the multiset sum of $I$ and $\{k\}$, which has size $n+1$). As this lemma involves differentiation, we need to differentiate the expressions from Definition \ref{DeltaJg}. For this we recall again that the derivative is of the expression $\Delta_{J}g\big(\vec{x},y(\vec{x})\big)$, where $y(\vec{x})$ is defined via $f(\vec{x},y)=0$ and therefore $y_{k}=-\frac{f_{k}}{f_{y}}$, and let $\eta_{j}$ denote the multiplicity with which $j$ appears in $J$.
\begin{lem}
$f_{y}^{2}\frac{\partial}{\partial x_{k}}\Delta_{J}g$ equals $f_{y}\Delta_{J+k}g+|J|\Delta_{k}f_{y}\cdot\Delta_{J}g-\sum_{j=1}^{N}\eta_{j}\Delta_{jk}f\cdot\Delta_{J \setminus j}g_{y}$. \label{derDelta}
\end{lem}
We have introduced the shorthand $J \setminus j$ for the difference $J\setminus\{j\}$, which is obtained, like the set difference, by subtracting 1 from the multiplicity of $j$ in $J$. The fact that this term involves the multiplier $\eta_{j}$ implies that it only involves indices $j$ that do appear in $J$, so that we do not have to consider negative multiplicities in $J \setminus j$. We note again that the last multiplier here is $\Delta_{J \setminus j}(g_{y})$, and not $(\Delta_{J \setminus j}g)_{y}$.

\begin{proof}
We let $f_{y}^{2}\frac{\partial}{\partial x_{k}}$ operate on the term associated with the subset $K \subseteq J$ in Definition \ref{DeltaJg}, and apply Leibniz' rule, to get 3 types of terms: One from acting on the derivative of $g$, one from differentiating the power of $f_{y}$, and one from the product $\prod_{i=1}^{N}f_{i}^{\kappa_{i}}$, where $\kappa_{i}$ is the multiplicity to which $i$ appears in $K$. The first term (associated with $K$) is
\begin{equation}
\textstyle{(-1)^{|K|}\binom{J}{K}\big[g_{[(J+k) \setminus K]y^{|K|}}\!\cdot\!\prod_{i}f_{i}^{\kappa_{i}}\!\cdot\!f_{y}^{|J|-|K|+2}\!-g_{(J \setminus K)y^{|K|}}\!\cdot\!\prod_{i}f_{i}^{\kappa_{i}}\!\cdot\!f_{k}f_{y}^{|J|-|K|+1}\big]}, \label{Kderg}
\end{equation}
and note that in the second summand we can write $J \setminus K=(J+k)\setminus(K+k)$, the product of the $f_{i}$'s with $f_{k}$ corresponds to the product associated with $K+k$ in Definition \ref{DeltaJg} and so does the extra sign, and $|J|-|K|+1$ equals $|J|-|K+k|+2$. Now, a multiset $L \subseteq J+k$ can either be presented as a multiset $K \subseteq J$ (if the multiplicity to which it contains $k$ is not already $\eta_{k}+1$) or as $K+k$ for a multiset $K \subseteq J$ (if $k$ does show up in this multiset), and then the first term in Definition \ref{DeltaJg} with $K=L$ (if $L \subseteq J$) and the second term there with $K=L \setminus k$ (if $k \in L$) give the same expression, with the multipliers $\binom{J}{L}$ and $\binom{J}{L \setminus k}$ respectively. But if the multiplicity to which $1 \leq i \leq N$ shows up in $L$ is $\lambda_{i}$ then the two terms are $\prod_{i \neq k}\binom{\eta_{i}}{\lambda_{i}}$ times $\binom{\eta_{k}}{\lambda_{k}}$ and $\binom{\eta_{k}}{\lambda_{k}-1}$ respectively. Thus their sum is $\prod_{i \neq k}\binom{\eta_{i}}{\lambda_{i}}$ times $\binom{\eta_{k}+1}{\lambda_{k}}$ (also when $\lambda_{k}$ is $\eta_{k}+1$ or 0), namely $\binom{J+k}{L}$, and as the power of $f_{y}$ is $|J+k|-|L|+1$, Definition \ref{DeltaJg} shows that this gives the first asserted term.

The other terms from differentiating the summand with index $K$, having multiplicities $\{\kappa_{i}\}_{i=1}^{N}$ (with the derivative acting either on some $f_{i}$ or on the power of $f_{y}$), combine to give $(-1)^{|K|}\binom{J}{K}g_{(J \setminus K)y^{|K|}}f_{y}^{|J|-|K|}$ times
\begin{equation}
\textstyle{\sum_{j=1}^{N}\kappa_{j}(f_{jk}f_{y}^{2}-f_{jy}f_{k}f_{y})\prod_{i=1}^{N}f_{i}^{\kappa_{i}-\delta_{i,j}}+(|J|-|K|)(f_{ky}f_{y}-f_{yy}f_{k})\prod_{i=1}^{N}f_{i}^{\kappa_{i}}}, \label{kderf}
\end{equation}
where $\delta_{i,j}$ in the exponent in the first term is the usual Kronecker $\delta$-symbol (and we get no negative exponents, since if $\kappa_{j}=0$ then the multiplier $\kappa_{j}$ annihilates the expression with $f_{j}^{-1}$). Since the expression in parentheses in the second term of Equation \eqref{kderf} equals $\Delta_{k}f_{y}$, taking the part that is multiplied by $|J|$, multiplying by the external coefficient, and summing over $K$ produces the second required term by Definition \ref{DeltaJg}. In the part that is multiplied by $|K|$, we replace this multiplier by $\sum_{j \in K}\kappa_{j}$, and extract the factor $f_{j}$ as well. This gives, for every $j$, the product $\kappa_{j}\prod_{i=1}^{N}f_{i}^{\kappa_{i}-\delta_{i,j}}$ times an expression which reduces to $\Delta_{jk}f$. Now, the latter product corresponds to the multiplicities of $K \setminus j$, the difference $J \setminus K$ is the complement of that set in $J \setminus j$ (with the same presentation of the difference of cardinalities in the exponent of $f_{y}$), and $\kappa_{j}\binom{J}{K}$ equals $\eta_{j}\binom{J \setminus j}{K \setminus j}$ because the binomial coefficients associated with $i \neq j$ are the same on both sides and for $j$ we use the well-known identity $\kappa_{j}\binom{\eta_{j}}{\kappa_{j}}=\eta_{j}\binom{\eta_{j}-1}{\kappa_{j}-1}$. Since the multisets that are contained in $J \setminus j$ are precisely the multisets $K \setminus j$ for a multiset $K \subseteq J$ which contains $j$, and the sign $(-1)^{|K \setminus j|}$ is $-(-1)^{|K|}$, summing over such $K \subseteq J$ indeed yields the remaining desired term via Definition \ref{DeltaJg}. This completes the proof of the lemma.
\end{proof}
The case of empty $J$ in Lemma \ref{derDelta} is just $f_{y}^{2}$ times the identity $\frac{\partial}{\partial x_{k}}g=g_{k}-\frac{g_{y}f_{k}}{f_{y}}$ from above, and the establishment of the last term in the proof of that lemma was based on proving that $\Delta_{jk}f$ can be obtained as $f_{y}\Delta_{k}f_{j}-f_{j}\Delta_{k}f_{y}$. This relation generalizes to the identity stating that $\Delta_{J+j}g=f_{y}\Delta_{J}g_{j}-f_{j}\Delta_{J}g_{y}$ for any multiset $J$ and index $j$, which one can prove using the fact that the subsets of $J+j$ are $K$ and $K+j$ for $K \subseteq J$ and manipulations of binomial coefficients like in the proof of Lemma \ref{derDelta}, but we omit the details since we shall not use this formula. The non-vanishing of Equation \eqref{killsym} in our evaluation of $f_{y}^{5}y_{ijk}$ is equivalent to the fact that the term $2\Delta_{k}f_{y}\cdot\Delta_{ij}f$ from Lemma \ref{derDelta} (with $g=f$ and $J$ of size 2) does not cancel with the terms $-\Delta_{ik}f\cdot\Delta_{j}f_{y}-\Delta_{jk}f\cdot\Delta_{i}f_{y}$, unlike the symmetric situation for ordinary derivatives, which gave a special case of cancelation happening for no other parameters. The fact that for $f_{y}^{5}y_{ijk}$ we obtain an expression that is symmetric in the indices (as we should) follows from the precise analysis of the coefficients.

We recall the value of $f_{y}^{5}y_{ijk}$ from Equation \eqref{yder3fin}, and wish to apply Lemmas \ref{denfree} and \ref{derDelta} in order to evaluate $f_{y}^{7}y_{ijkl}$, where $i$, $j$, $k$ and $l$ are distinct indices (so that no binomial coefficients are involved). We begin with the parts arising from $-f_{y}\Delta_{ijk}f$, where for the action of $f_{y}^{2}\frac{\partial}{\partial x_{l}}$ we apply Leibniz' rule, yielding $-f_{y}\Delta_{l}f_{y}\cdot\Delta_{ijk}f$ from the differentiation of $f_{y}$, plus $-f_{y}$ times the expression \[f_{y}\Delta_{ijkl}f+3\Delta_{l}f_{y}\cdot\Delta_{ijk}f-\Delta_{il}f\cdot\Delta_{jk}f_{y}-\Delta_{jl}f\cdot\Delta_{ik}f_{y}-\Delta_{kl}f\cdot\Delta_{ij}f_{y}\] from Lemma \ref{derDelta}, and the term $+5\Delta_{l}f_{y}\Delta_{ijk}f$ from the second summand in Lemma \ref{denfree}. These combine to \[-f_{y}^{2}\Delta_{ijkl}f+\Delta_{l}f_{y}\cdot\Delta_{ijk}f+f_{y}\Delta_{il}f\cdot\Delta_{jk}f_{y}+f_{y}\Delta_{jl}f\cdot\Delta_{ik}f_{y}+f_{y}\Delta_{kl}f\cdot\Delta_{ij}f_{y}.\] Doing the same with the summand $\Delta_{k}f_{y}\cdot\Delta_{ij}f$ from Equation \eqref{yder3fin} produces the sum of the expressions $[f_{y}\Delta_{kl}f_{y}+\Delta_{l}f_{y}\cdot\Delta_{k}f_{y}-\Delta_{kl}f \cdot f_{yy}]\Delta_{ij}f$, \[\Delta_{k}f_{y}[f_{y}\Delta_{ijl}f+2\Delta_{l}f_{y}\cdot\Delta_{ij}f-\Delta_{il}f\cdot\Delta_{j}f_{y}-\Delta_{jl}f\cdot\Delta_{i}f_{y}],\] and $-5\Delta_{l}f_{y}\cdot\Delta_{k}f_{y}\cdot\Delta_{ij}f$, so that the total coefficient of $\Delta_{l}f_{y}\cdot\Delta_{k}f_{y}\cdot\Delta_{ij}f$ becomes $-2$. Applying the same procedure with the other two terms from Equation \eqref{yder3fin} gives similar expressions (with $k$ interchanged with $j$ or with $i$), where in particular each of the summands $\Delta_{k}f_{y}\cdot\Delta_{il}f\cdot\Delta_{j}f_{y}$ and $\Delta_{k}f_{y}\cdot\Delta_{jl}f\cdot\Delta_{i}f_{y}$ shows up again in one of these expressions. Gathering all these terms, we obtain that $f_{y}^{7}y_{ijkl}$ equals $-f_{y}^{2}\Delta_{ijkl}f$ plus
\begin{equation}
[f_{y}\Delta_{l}f_{y}\cdot\Delta_{ijk}f]^{\mathrm{s}}+[f_{y}\Delta_{ij}f_{y}\cdot\Delta_{kl}f]^{\mathrm{s}}-[f_{yy}\Delta_{ij}f\cdot\Delta_{kl}f]^{\mathrm{s}}-2[\Delta_{k}f_{y}\cdot\Delta_{l}f_{y}\cdot\Delta_{ij}f]^{\mathrm{s}}, \label{yder4}
\end{equation}
where the superscript $\mathrm{s}$ means adding to it all the distinct elements required for getting an expression that is symmetric under any permutation of the indices $i$, $j$, $k$, and $l$. When some of the indices $i$, $j$, $k$, and $l$ coincide, some of the terms inside each such ``$\mathrm{s}$-orbit'' coincide, so that we obtain less terms with higher multiplicities, and the sum of the multiplicities is always 4, 6, 3, and 6 for the respective summands from Equation \eqref{yder4} (this, times the coefficient 2 in the last summand, is in correspondence with the coefficients in Equation (7) of \cite{[Z]}, which we obtain again when all the indices coincide and we differentiate with respect to a single variable).

The explicit formula for $y_{I}$, or equivalently $f_{y}^{2n-1}y_{I}$ when $n=|I|\geq2$, will be established, as in \cite{[Z]}, in two steps. First we shall determine which expressions can show up in it, and only later we will find the explicit coefficient with which every expression does show up. The fact that the expression is symmetric under permuting the elements of $I$ will be a consequence of the second step only, since the coefficients in question will have this symmetry property.

For the first step, we begin by considering the examples that we already have. When $n=2$ we saw that $f_{y}^{3}y_{I}=-\Delta_{I}f$, consisting of one multiplier, with index $I$, and no external derivative with respect to $y$. The expression from Equation \eqref{yder3fin} shows that for $n=3$ we can write $f_{y}^{5}y_{I}$ as $-f_{y}\Delta_{I}f+\sum_{i \in I}\Delta_{i}f_{y}\cdot\Delta_{I \setminus i}f$ (the sum over $i$ is with multiplicities when $I$ is a multiset, and should be interpreted in general as $\sum_{i=1}^{N}\nu_{i}\Delta_{i}f_{y}\cdot\Delta_{I \setminus i}f$), in which each summand is the product of two terms, one of which involves $f$ and the other $f_{y}$, and the union of the indices is $I$. In case $n=4$, the terms from Equation \eqref{yder4} for a set $I$ without multiplicities are $-f_{y}^{2}\Delta_{I}f$, $\sum_{i \in I}f_{y}\Delta_{i}f_{y}\cdot\Delta_{I \setminus i}f$ (with a plus sign), $\sum_{J}f_{y}\Delta_{J}f_{y}\cdot\Delta_{I \setminus J}f$ for $J \subseteq I$ of size 2, the sum over partitions of $I$ into unmarked sets $J_{1}$ and $J_{2}$ of size 2 of $-f_{yy}\Delta_{J_{1}}f\cdot\Delta_{J_{2}}f$, and $-2\sum_{J}f_{y}\Delta_{I \setminus J}f\cdot\prod_{i \in J}\Delta_{i}f_{y}$, the sum of which again taken over $J \subseteq I$ of size 2 (the latter can be seen as partitions of $I$ into a set $K$ of size 2 and two unmarked sets $J_{1}$ and $J_{2}$ of size 1, with the expression being $-2\Delta_{J_{1}}f_{y}\Delta_{J_{2}}f_{y}\Delta_{K}f$). When $I$ is a multiset of size 4, the same description is valid, with sums involving the corresponding multiplicities, and we shall not write these explicitly at this point. Each such product is of three expressions, the total number of $y$-indices is 2, and the sets in the $\Delta$-indices forming a partition of $I$, where in partitions involving unmarked sets, the $\Delta$-operators of the unmarked sets of same type act on the same function (i.e., on $f$ with the same amount of $y$-indices). The general form of a term in the expression of order $n$ is thus as follows.
\begin{prop}
If $|I|=n\geq2$ then $f_{y}^{2n-1}y_{I}$ is a linear combination of expressions of the sort $\prod_{e=1}^{n-1}\Delta_{J_{e}}f_{y^{r_{e}}}$, satisfying the following conditions: In each such term the multisets $J_{e}$ sum to  $I$ as multisets (some $J_{e}$'s can be empty), we have $\sum_{e=1}^{n-1}r_{e}=n-2$, and the expressions with $r_{e}=0$ cannot show up with $J_{e}$ empty or a singleton. \label{formnoden}
\end{prop}

\begin{proof}
For $n=2$ the single term $\Delta_{I}f$ satisfies the required properties (and one can check that so do the terms showing up in our explicit expressions when $n$ is 3 or 4), and we assume, by induction, that it holds for a multiset $I$ of cardinality $n$. We consider a multiset of cardinality $n+1$ which is of the form $I+k$ for some index $k$, and then Lemma \ref{denfree} expresses $f_{y}^{2n-1}y_{I}$ as the sum of two terms, the second of which is a constant multiple of $\Delta_{k}f_{y}$ times the original expression $f_{y}^{2n-1}y_{I}$. Since this multiplies every product from $f_{y}^{2n-1}y_{I}$ by one expression, with the extra index $k$ and with one additional index $y$, we indeed obtain a linear combination of products of $n$ terms, with a total of $n-1$ indices of $y$, and with $\Delta$-indices whose multiset sum is $I+k$. For the first term we let $f_{y}^{2}\frac{\partial}{\partial x_{k}}$ act via Leibniz' rule on every product from $f_{y}^{2n-1}y_{I}$, yielding products of $n-2$ elements times the derivative that we evaluated in Lemma \ref{derDelta}. Each such evaluation yields three terms, the second of which is the original expression times $\Delta_{k}f_{y}$, with which we have already dealt. The first term from Lemma \ref{derDelta} adds the index $k$ to the corresponding multiset $J_{e}$ (thus to the multiset sum $I$), and with the extra term $f_{y}$, having the empty multiset in the $\Delta$-index and one index $y$ more, we again get elements of the desired sort. Finally, for each $1 \leq j \leq N$ whose multiplicity in $J_{e}$ is positive, the third term from Lemma \ref{derDelta} replaces $\Delta_{J_{e}}f_{y^{r_{e}}}$ by two terms (increasing the number of multipliers by 1 yet again), one with $j$ missing from $J_{e}$ but the additional index $y$, and the other one, $\Delta_{jk}f$, with the remaining indices to complete to the multiset sum $I+k$ and no $y$-index. As none of the expressions $f$ or $\Delta_{i}f$ shows up in this process, this proves the proposition.
\end{proof}

The proof of Proposition \ref{formnoden} yields, in its induction step, either a multiplier of $f_{y}$ (having empty $J_{e}$ and $r_{e}=1$), or a multiplier $\Delta_{k}f_{y}$ (with $J_{e}$ a singleton and $r_{e}=1$), or a multiplier $\Delta_{jk}f$ (whose multiset $J_{e}$ has size 2 but $r_{e}=0$). The fact that every combination from that proposition must indeed contain one such multiplier essentially follows from Proposition 5 of \cite{[Z]} using cardinalities, but we reproduce the argument in our setting. Note that in \cite{[Z]} we considered coordinate-wise sums of vectors of the form $\sqbinom{l}{r}$ with $l$ and $r$ in $\mathbb{N}$. Here we have sets and numbers, but for keeping the notation we shall write $\sqbinom{J}{r}+\sqbinom{K}{s}$ for $\sqbinom{J+K}{r+s}$ where $J$ and $K$ are multisets of indices between 1 and $N$, $J+K$ is their multiset sum, and $r$ and $s$ are non-negative integers. The conditions from Proposition \ref{formnoden} mean that each term corresponds to an unordered collection of $n-1$ ``vectors'' $\sqbinom{J_{e}}{r_{e}}$, with $\sum_{e=1}^{n-1}\sqbinom{J_{e}}{r_{e}}=\sqbinom{I}{n-2}$.
\begin{prop}
Assume that we are given, for every $1 \leq e \leq n-1$, a pair $\sqbinom{J_{e}}{r_{e}}$, and that if $r_{e}=0$ then the multiset $J_{e}$ is neither empty nor a singleton. Assume further that the vectors $\sqbinom{J_{e}}{r_{e}}$ sum to $\sqbinom{I}{n-2}$. Then either there exists $e$ with $J_{e}$ empty and $r_{e}=1$, or there is $e$ with $|J_{e}|=1$ and $r_{e}=1$, or we have $|J_{e}|=2$ and $r_{e}=0$ for some $e$. \label{genrec}
\end{prop}

\begin{proof}
The proof follows that of Proposition 5 of \cite{[Z]}. Note that the assumption on the sum yields $\sum_{e=1}^{n-1}|J_{e}|+\sum_{e=1}^{n-1}r_{e}=|I|+n-2=2n-2$. Assume now that the first case does not happen (otherwise we are done). Then we have $|J_{e}|+r_{e}\geq2$ for every $1 \leq e \leq n-1$ since the other two cases in which this inequality does not hold are not allowed. Summing over $e$ we obtain that $\sum_{e=1}^{n-1}|J_{e}|+\sum_{e=1}^{n-1}r_{e}\geq2n-2$, but as we know that the left hand side equals $2n-2$, we deduce that all the inequalities are equalities. The case where $J_{e}$ is empty and $r_{e}=2$ for every $e$ would yield a sum of $\sqbinom{\emptyset}{2n-2}$, and as this is different from $\sqbinom{I}{n-2}$, this situation cannot occur. Combining this with the sum 2 property yields that either $|J_{e}|=1$ and $r_{e}=1$ for some $e$, or there is $e$ with $|J_{e}|=2$ and $r_{e}=0$, as desired. This proves the proposition.
\end{proof}

Proposition \ref{genrec} does not yet assure us that any product satisfying the conditions from Proposition \ref{formnoden} indeed shows up in the formula for $f_{y}^{2n-1}y_{I}$ (this will be a consequence of the not-vanishing of the coefficients from Theorem \ref{coeffval} below). But we do see it in the cases where we already have the complete formula. When $n=2$ the only possible product of a single element corresponds to $\sqbinom{I}{0}$, and in case $n=3$ we can decompose $\sqbinom{I}{1}$, up to the order of the summands, as either $\sqbinom{I}{0}+\sqbinom{0}{1}$, or $\sqbinom{I \setminus i}{0}+\sqbinom{i}{1}$ for $i \in I$ (where we write $i$ in case the set $J_{e}$ is the singleton $\{i\}$). For $n=4$ the unordered sums producing $\sqbinom{I}{2}$ are $\sqbinom{I}{0}+\sqbinom{0}{1}+\sqbinom{0}{1}$, $\sqbinom{I \setminus i}{0}+\sqbinom{i}{1}+\sqbinom{0}{1}$ for each $i \in I$, $\sqbinom{J}{1}+\sqbinom{I \setminus J}{0}+\sqbinom{0}{1}$ for a multiset $J \subseteq I$ of cardinality 2, the combination $\sqbinom{J_{1}}{0}+\sqbinom{J_{2}}{0}+\sqbinom{0}{2}$ where $J_{1}+J_{2}$ is an unordered partition of $I$ into two multisets of cardinality 2, and $\sqbinom{I\setminus\{i,j\}}{0}+\sqbinom{i}{1}+\sqbinom{j}{1}$ for two indices $i$ and $j$ with $i$ and $j$ in $I$ when they are distinct or with $i=j$ of multiplicity at least 2 in $I$ (we write $I\setminus\{i,j\}$ and not $I \setminus ij$ to avoid the ambiguity with $(I \setminus i)j$, which may look like $(I \setminus i)+j$). Comparing these decompositions with the expressions for the derivatives given in Equations \eqref{yder2}, \eqref{yder3fin}, and \eqref{yder4} exemplifies that the terms in the formula for $f_{y}^{2n-1}y_{I}$ are expected to be in one-to-one correspondence with the products satisfying the conditions from Propositions \ref{formnoden} and \ref{genrec}.

Note that the expression that we seek is $y_{I}$, and not the product $f_{y}^{2n-1}y_{I}$, so that we consider the following modification. As the vectors $\sqbinom{J_{e}}{0}$ with $|J_{e}|\leq1$ are excluded, every vector $\sqbinom{J_{e}}{r_{e}}$ with $|J_{e}|+r_{e}<2$ is $\sqbinom{\emptyset}{1}$ and corresponds to $f_{y}$, which is expected to show up in the denominator of the formula for $y_{I}$ rather than as one of the multipliers. In order to be able to write the expressions for both $f_{y}^{2n-1}y_{I}$ and $y_{I}$, we recall that partitions can be given in terms of multiplicities, and make the following definition.
\begin{defn}
Let a set $I$ be given, with multiplicities $\{\nu_{i}\}_{i=1}^{N}$, and assume that $n=|I|=\sum_{i=1}^{N}\nu_{i}\geq2$. Then $\tilde{A}_{I}$ stands for the set of partitions of $\sqbinom{I}{n-2}$ having $n-1$ (unordered) parts, none of which is of the form $\sqbinom{\emptyset}{0}$ or $\sqbinom{i}{0}$ for a singleton $\{i\}$. We shall write elements $\tilde{\alpha}\in\tilde{A}_{I}$ using multiplicities $m_{J,r}\geq0$ for multisets $J$ and integers $r\geq0$, with $m_{0,0}=m_{i,0}=0$ for any $i \in I$ (where the index $i$ stands for $J=\{i\}$), and such that the equalities $\sum_{J,r}m_{J,r}=n-1$ and $\sum_{J,r}rm_{J,r}=n-2$ of numbers hold, as well as the multiset sum equality $\sum_{J,r}m_{J,r}J=I$. We also define $A_{I}$ to be the set of multiplicities $\{m_{J,r}\}_{\{|J|+r\geq2\}}$ for which we have $\sum_{J,r}m_{J,r}J=I$ as multisets and $\sum_{J,r}(r-1)m_{J,r}=-1$. \label{AIdef}
\end{defn}

Like in Lemma 7 of \cite{[Z]}, we shall need the essential finiteness of elements of the sets $\tilde{A}_{I}$ and $A_{I}$ from Definition \ref{AIdef}, as well as a natural identification between these sets.
\begin{lem}
For any multiset $I$ of size $n\geq2$, and every element of $\tilde{A}_{I}$ or of $A_{I}$, we have $m_{J,r}=0$ for all but finitely many vectors $\sqbinom{J}{r}$. Moreover, there is a canonical isomorphism between $\tilde{A}_{I}$ and $A_{I}$. \label{AItildeAI}
\end{lem}

\begin{proof}
Note that for an element $\{m_{J,r}\}_{J,r}$ of either set, the multiset equality yields $\sum_{J,r}|J|m_{J,r}=n$. Thus, if $\{m_{J,r}\}_{\{|J|+r\geq2\}}$ is an element of $A_{I}$ then the inequality condition on the indices and the fact that multiplicities are non-negative bounds $\sum_{\{|J|+r\geq2\}}m_{J,r}$ by $\sum_{\{|J|+r\geq2\}}(|J|+r-1)m_{J,r}$, which equals $n-1$ by the latter equality and Definition \ref{AIdef}. The fact that we also have $\sum_{J,r}m_{J,r}=n-1$ for elements of $\tilde{A}_{I}$ yields the first assertion for elements of both sets. The isomorphism between the two sets is defined, as in \cite{[Z]}, by taking the element $\tilde{\alpha}=\{m_{J,r}\}_{J,r}\in\tilde{A}_{I}$ to $\alpha=\{m_{J,r}\}_{\{|J|+r\geq2\}}$, which means forgetting the multiplicity $m_{\emptyset,1}$ since $m_{\emptyset,0}$ and $m_{i,0}$ for singletons $\{i\}$ vanish by definition. Subtracting the two general equalities defining elements of $\tilde{A}_{I}$ in Definition \ref{AIdef} shows that such elements satisfy $\sum_{J,r}(r-1)m_{J,r}=-1$, and since $m_{\emptyset,1}$ does not contribute to any multiset sum as well as to the sum $\sum_{J,r}(r-1)m_{J,r}$, our map indeed takes $\tilde{A}_{I}$ to $A_{I}$. Moreover, given an element of $A_{I}$, in any pre-image of our element in $\tilde{A}_{I}$, the multiplicity $m_{J,r}$ is known when $|J|+r\geq2$ and we must have $n-1=\sum_{J,r}m_{J,r}=\sum_{\{|J|+r\geq2\}}m_{J,r}+m_{\emptyset,1}$, which determines $m_{\emptyset,1}$ and shows that the map is injective. Moreover, as we saw that $\sum_{\{|J|+r\geq2\}}m_{J,r} \leq n-1$ for every element of $A_{I}$, the multiplicity $m_{\emptyset,1}$ that we must add for obtaining our pre-image is non-negative, yielding the surjectivity of our map as well. The canonical map $\tilde{\alpha}\mapsto\alpha$ that we constructed is thus a bijection. This proves the lemma.
\end{proof}

We can therefore obtain the initial form of the formula for $y_{I}$.
\begin{cor}
There exist coefficients $\{c_{\alpha}\}_{\alpha \in A_{I}}$ such that $y_{I}$ is a sum of the form $\sum_{\alpha=\{m_{J,r}\}_{\{|J|+r\geq2\}} \in A_{I}}\Big[c_{\alpha}\prod_{|J|+r\geq2}(\Delta_{J}f_{y^{r}})^{m_{J,r}}\Big/f_{y}^{n+\sum_{|J|+r\geq2}m_{J,r}}\Big]$. \label{formwden}
\end{cor}

\begin{proof}
Using Definition \ref{AIdef}, we deduce from Propositions \ref{formnoden} and \ref{genrec} that $y_{I}$ can be presented as the sum $\sum_{\tilde{\alpha}=\{m_{J,r}\}_{J,r}\in\tilde{A}_{I}}c_{\tilde{\alpha}}\prod_{J,r}(\Delta_{J}f_{y^{r}})^{m_{J,r}}\big/f_{y}^{2n-1}$. We replace, using Lemma \ref{AItildeAI}, each $\tilde{\alpha}\in\tilde{A}_{I}$ by the corresponding $\alpha \in A_{I}$, denote $c_{\tilde{\alpha}}$ also by $c_{\alpha}$, and merge $(\Delta_{\emptyset}f_{y^{1}})^{m_{\emptyset,1}}/f_{y}^{2n-1}$ into $f_{y}^{m_{\emptyset,1}-2n+1}$. Recalling the value of $m_{\emptyset,1}$ that we attached to any $\alpha \in A_{I}$ in the proof of that lemma, this produces the desired expression. This proves the corollary.
\end{proof}
We remark again that the coefficient $c_{\alpha}$ from Corollary \ref{formwden} must remain the same when $\alpha$ is replaced by its image under a permutation of the indices that leaves the multiset $I$ invariant, since the formula for $y_{I}$ must be invariant under such permutations. This does not follow from the proof of that corollary, and will be established only after we determine the formula for these coefficients in Theorem \ref{coeffval} below.

\begin{rmk}
For every $\alpha \in A_{I}$, the sum $h=\sum_{|J|+r\geq2}m_{J,r}$ is again the number of vectors in the partition, but the partition now is of $\sqbinom{I}{h-1}$ (into $h$ vectors). As with $m_{J,r}=0$ for every $J$ and $r$ with $|J|+r\geq2$ the equalities defining an element of $A_{I}$ in Definition \ref{AIdef} are not satisfied, we deduce that $h$ cannot vanish, and since $2h=\sum_{|J|+r\geq2}2m_{J,r}\leq\sum_{|J|+r\geq2}(|J|+r)m_{J,r}$, and we saw in the proof of Lemma \ref{AItildeAI} that the latter sum equals $h+n-1$, we deduce that $1 \leq h \leq n-1$. It will thus be convenient to write $A_{I}=\bigcup_{h=1}^{n-1}A_{I,h}$ (a disjoint union), where $A_{I,h}$ is obtained from $A_{I}$ by adding the additional constraint that $\sum_{|J|+r\geq2}m_{J,r}=h$, which is equivalent to $\sum_{|J|+r\geq2}(r-1)m_{J,r}=h-1$. Then in Corollary \ref{formwden} the summands that arise from $A_{I,h}$ involve the denominator $f_{y}^{n+h}$. Returning to $A_{I}$, we see that $I$ only shows up in the multiset equality, and we define $A$ to be the set of all $\alpha=\{m_{J,r}\}_{\{|J|+r\geq2\}}$ that satisfy the non-negativity assumption $m_{J,r}\geq0$ for every $J$ and $r$, the finiteness of $\{(J,r)|m_{J,r}>0\}$ from Lemma \ref{AItildeAI}, and the equality $\sum_{J,r}(r-1)m_{J,r}=-1$ from Definition \ref{AIdef}. It is clear that $\bigcup_{|I|\geq2}A_{I}$ is a disjoint union that is contained in $A$ (since $I$ is determined as $\sum_{J,r}m_{J,r}J$, which is a finite sum defining a finite multiset), and we claim that this union gives all of $A$. Indeed, for an element of $A$ we get that $\sum_{|J|+r\geq2}m_{J,r}|J|$ equals $\sum_{|J|+r\geq2}(|J|+r-1)m_{J,r}+1$ (by the defining equality from Definition \ref{AIdef}), which is at least $\sum_{|J|+r\geq2}m_{J,r}+1$ by non-negativity and the index restriction $|J|+r\geq2$. But as the latter sum was seen to be $h\geq1$, we find that by setting $I$ to be the multiset sum $\sum_{|J|+r\geq2}m_{J,r}J$ we get $|I|\geq2$, and therefore every element of $A$ lies in a (unique) set $A_{I}$ with $|I|\geq2$. \label{setpart}
\end{rmk}

\section{The Combinatorial Coefficients \label{DetCoeff}}

In order to be able to calculate the coefficients $c_{\alpha}$ from Corollary \ref{formwden}, we shall construct, like in \cite{[Z]}, a recursive relation for them. From the formulae that we have already established, it follows that when $|I|=2$ and $\alpha$ is the only element lying in $A_{2}$, this coefficient is $-1$, and when $|I|=3$ we get again a coefficient of $-1$ for the element of $A_{3}$ with $\sqbinom{I}{0}$, while for the three other elements there it equals $+1$ when the elements of $I$ are distinct (if $I$ is a multiset with multiplicities 2 and 1, then these become two terms with the coefficients $+2$ and $+1$, and when $I$ is a single element having multiplicity 3, these elements of $A_{3}$ all merge to a single one, with the coefficients $+3$). When $I$ is a set of size 4 (meaning a multiset all of whose elements are distinct), we can divide Equation \eqref{yder4} by $f_{y}^{7}$ and read the coefficients off the elements of $A_{4}$ from there.

For establishing the recursive relation, consider a multiset $I$ of size $n$, and in the corresponding set $\tilde{A}_{I}$ from Corollary \ref{formwden} we choose an element $\tilde{\alpha}$, which determines the multiplicities $\{m_{J,r}\}_{J,r}$. Given some index $k$, we define the following objects. First, take a vector $\sqbinom{\tilde{J}}{\tilde{r}}$ satisfying $\tilde{J}+\tilde{r}\geq2$ and $m_{\tilde{J},\tilde{r}}\geq1$, and given $J$ and $r$ we set the multiplicity $m_{+,J,r}^{k,\tilde{J},\tilde{r}}$ to be $m_{J,r}+1$ in case $\sqbinom{J}{r}$ is $\sqbinom{\emptyset}{1}$ or $\sqbinom{\tilde{J}+k}{\tilde{r}}$, to be $m_{J,r}-1$ when $\sqbinom{J}{r}=\sqbinom{\tilde{J}}{\tilde{r}}$, and simply $m_{J,r}$ in any other case. We define $\tilde{\alpha}_{k,+}^{\tilde{J},\tilde{r}}$ to be the set of multiplicities $\{m_{+,J,r}^{k,\tilde{J},\tilde{r}}\}_{J,r}$. Assume now that $\sqbinom{\tilde{J}}{\tilde{r}}$ is with $\tilde{J}\neq\emptyset$ and with $m_{\tilde{J},\tilde{r}}\geq1$, and take $j\in\tilde{J}$, with multiplicity $\tilde{\eta}_{j}\geq1$ in $\tilde{J}$, such that if $r=0$ then $\tilde{J}$ (which cannot be a singleton) is not $\{j,k\}$ with our index $k$. In this situation we define $m_{t,J,r}^{k,\tilde{J},\tilde{r},j}$ to be $m_{J,r}+1$ if $\sqbinom{J}{r}$ equals either $\sqbinom{jk}{0}$ or $\sqbinom{\tilde{J} \setminus j}{\tilde{r}+1}$, $m_{J,r}-1$ in case $\sqbinom{J}{r}=\sqbinom{\tilde{J}}{\tilde{r}}$, and $m_{J,r}$ otherwise. We then set $\tilde{\alpha}_{k,t}^{\tilde{J},\tilde{r},j}:=\{m_{t,J,r}^{k,\tilde{J},\tilde{r},j}\}_{J,r}$. Moreover, the multiplicity $m_{m,J,r}^{k}$ is defined to be $m_{J,r}+1$ when $\sqbinom{J}{r}=\sqbinom{k}{1}$ and $m_{J,r}$ whenever $\sqbinom{J}{r}\neq\sqbinom{k}{1}$, and we denote $\{m_{m,J,r}^{k}\}_{J,r}$ by $\tilde{\alpha}_{k,m}$. Using these definitions, we obtain the following lemma.
\begin{lem}
If $\tilde{\alpha}$ is in $\tilde{A}_{I}$, $1 \leq k \leq N$, and $\tilde{J}$, $\tilde{r}$, and $j$ are such that following the symbols are defined, then $\tilde{\alpha}_{k,+}^{\tilde{J},\tilde{r}}$, $\tilde{\alpha}_{k,t}^{\tilde{J},\tilde{r},j}$, and $\tilde{\alpha}_{k,m}$ are all elements of $\tilde{A}_{I+k}$. Consider now the corresponding expression $c_{\tilde{\alpha}}\prod_{J,r}(\Delta_{J}f_{y^{r}})^{m_{J,r}}$, and let $f_{y}^{2}\frac{\partial}{\partial x_{k}}-(2n-1)\Delta_{k}f_{y}$, with $n:=|I|$, act on it. The result gives the combination with the following coefficients: If $\tilde{J}$ and $\tilde{r}$ are such that $|\tilde{J}|+\tilde{r}\geq2$, then the product corresponding to $\tilde{\alpha}_{k,+}^{\tilde{J},\tilde{r}}$ appears multiplied by $m_{\tilde{J},\tilde{r}}c_{\tilde{\alpha}}$; When $j$ is in $\tilde{J}$ with multiplicity $\tilde{\eta}_{j}$ and $\sqbinom{\tilde{J}}{\tilde{r}}\neq\sqbinom{jk}{0}$, the product arising from $\tilde{\alpha}_{k,t}^{\tilde{J},\tilde{r},j}$ comes with the coefficient $-\tilde{\eta}_{j}m_{\tilde{J},\tilde{r}}c_{\tilde{\alpha}}$; And there is one additional term, which is the product associated with $\tilde{\alpha}_{k,m}$ times the coefficient $-\big(n-1-m_{\emptyset,1}+\sum_{j=1}^{N}(1+\delta_{j,k})m_{jk,0}\big)c_{\tilde{\alpha}}$. \label{contcalpha}
\end{lem}

\begin{proof}
We evaluate the action of $f_{y}^{2}\frac{\partial}{\partial x_{k}}$ on $c_{\tilde{\alpha}}\prod_{J,r}(\Delta_{J}f_{y^{r}})^{m_{J,r}}$ via Leibniz' rule, and the action on the multiplier with $\tilde{J}$ and $\tilde{r}$ yields the coefficient $m_{\tilde{J},\tilde{r}}$ from the exponent and replaces this multiplier by the expression from Lemma \ref{derDelta}. The first term from that lemma yields, wherever $|\tilde{J}|+\tilde{r}\geq2$, the asserted term corresponding to $\tilde{\alpha}_{k,+}^{\tilde{J},\tilde{r}}$, as we have the vector $\sqbinom{\tilde{J}+k}{\tilde{r}}$ and the multiplier $f_{y}$ instead of one instance of $\sqbinom{\tilde{J}}{\tilde{r}}$. Considering now the $j$th summand in the third term from that lemma, in the action on the multiplier with $\tilde{J}\neq\emptyset$, $j\in\tilde{J}$, and $\tilde{r}$ such that $\sqbinom{\tilde{J}}{\tilde{r}}\neq\sqbinom{jk}{0}$, the vector $\sqbinom{\tilde{J}}{\tilde{r}}$ is replaced by $\sqbinom{\tilde{J} \setminus j}{\tilde{r}+1}$, and we have the multiplier $\Delta_{jk}f$, increasing the multiplicity $m_{jk,0}$ by 1.

It remains to consider the second term from Lemma \ref{denfree}, the contribution of the second term from Lemma \ref{derDelta} for every $\tilde{J}$ and $\tilde{r}$, the first term from Lemma \ref{derDelta} from the action on the power of $f_{y}$ associated with $\sqbinom{\tilde{J}}{\tilde{r}}=\sqbinom{\emptyset}{1}$, and the third term arising from $j$ when the action is on the multiplier $\Delta_{jk}f$ and $\sqbinom{\tilde{J}}{\tilde{r}}=\sqbinom{jk}{0}$. Every such expression produces the original expression times $\Delta_{k}f_{y}$, and therefore give a multiple of the product associated with $\tilde{\alpha}_{k,m}$. The coefficients with which this product appears in the terms thus described are $-(2n-1)c_{\tilde{\alpha}}$, the sum over $J$ and $r$ of $|J|m_{J,r}c_{\tilde{\alpha}}$ (which combine to $+nc_{\tilde{\alpha}}$ because $\tilde{\alpha}\in\tilde{A}_{I}$, as Definition \ref{AIdef} implies), $+m_{\emptyset,1}c_{\tilde{\alpha}}$, and the sum over $j$ of $-m_{jk,0}c_{\tilde{\alpha}}$ times the multiplicity of $j$ in the multiset $\{j,k\}$, which is $1+\delta_{j,k}$. This gives the asserted expression, which also establishes, together with Proposition \ref{formnoden}, the statement that all of these sets of multiplicities lie in $\tilde{A}_{I+k}$ (the latter is also easy to verify directly). This proves the lemma.
\end{proof}

Lemma \ref{contcalpha} evaluates the action of $f_{y}^{2}\frac{\partial}{\partial x_{k}}-(2n-1)\Delta_{k}f_{y}$ on a single summand, associated with an element of $\tilde{A}_{I}$. In order to find the total coefficient in front of the summand corresponding to an element of $\tilde{A}_{I+k}$ when $f_{y}^{2}\frac{\partial}{\partial x_{k}}-(2n-1)\Delta_{k}f_{y}$ acts on a combination of many summands from $\tilde{A}_{I}$, we need the dual notation to that used in Lemma \ref{contcalpha}. Take thus an element $\beta=\{\mu_{J,r}\}_{\{|J|+r\geq2\}}$ from $A_{I+k}$ (rather than $\tilde{A}_{I+k}$). Consider first a vector $\sqbinom{\hat{J}}{\hat{r}}$ with $k\in\hat{J}$, $|\hat{J}|+\hat{r}\geq3$, and $\mu_{\hat{J},\hat{r}}\geq1$, and set $\mu_{-,J,r}^{k,\hat{J},\hat{r}}$ to be $\mu_{J,r}-1$ in case $\sqbinom{J}{r}=\sqbinom{\hat{J}}{\hat{r}}$, $\mu_{J,r}+1$ when $\sqbinom{J}{r}=\sqbinom{\hat{J} \setminus k}{\hat{r}}$, and $\mu_{J,r}$ otherwise. The resulting element $\{\mu_{-,J,r}^{k,\hat{J},\hat{r}}\}_{\{|J|+r\geq2\}}$ will be denoted by $\beta_{k,-}^{\hat{J},\hat{r}}$. Now take $1 \leq j \leq N$ such that $\mu_{jk,0}\geq1$ and take $\sqbinom{\hat{J}}{\hat{r}}$ satisfying $\hat{r}\geq1$, $|\hat{J}|+\hat{r}\geq2$, $\sqbinom{\hat{J}}{\hat{r}}\neq\sqbinom{k}{1}$, and again $\mu_{\hat{J},\hat{r}}\geq1$, and we define $\mu_{b,J,r}^{k,\hat{J},\hat{r},j}$ to be $\mu_{J,r}-1$ when $\sqbinom{J}{r}$ equals either $\sqbinom{\hat{J}}{\hat{r}}$ or $\sqbinom{jk}{0}$, $\mu_{J,r}+1$ if $\sqbinom{J}{r}=\sqbinom{\hat{J}+j}{\hat{r}-1}$, and $\mu_{J,r}$ in any other case. We then set $\beta_{k,b}^{\hat{J},\hat{r},j}$ to be the element with multiplicities $\{\mu_{b,J,r}^{k,\hat{J},\hat{r},j}\}_{\{|J|+r\geq2\}}$. Finally, given $\beta$ for which $\mu_{k,1}\geq1$, the multiplicity $\mu_{d,J,r}^{k}$ is defined to be $\mu_{J,r}-1$ when $\sqbinom{J}{r}=\sqbinom{k}{1}$ and $\mu_{J,r}$ in case $\sqbinom{J}{r}\neq\sqbinom{k}{1}$, and the element $\{\mu_{d,J,r}^{k}\}_{\{|J|+r\geq2\}}$ is denoted by $\beta_{k,d}$. Recalling that adding or omitting the tilde corresponds to the map from Lemma \ref{AItildeAI} (in the appropriate direction), we establish the following duality between these constructions.
\begin{lem}
Given a multiset $I$ with $|I|\geq2$, an index $1 \leq k \leq N$, and an element $\beta \in A_{I+k}$, if any of the expressions $\beta_{k,-}^{\hat{J},\hat{r}}$, $\beta_{k,b}^{\hat{J},\hat{r},j}$, or $\beta_{k,d}$ is defined, then it lies in $A_{I}$. Take also an element $\tilde{\alpha}\in\tilde{A}_{I}$, and two vectors $\sqbinom{\tilde{J}}{\tilde{r}}$ and $\sqbinom{\hat{J}}{\hat{r}}$. Under the assumption that $\sqbinom{\hat{J}}{\hat{r}}=\sqbinom{\tilde{J}+k}{\tilde{r}}$ (which is equivalent to $k\in\hat{J}$ and $\sqbinom{\tilde{J}}{\tilde{r}}=\sqbinom{\hat{J} \setminus k}{\hat{r}}$), we have that $\tilde{\alpha}_{k,+}^{\tilde{J},\tilde{r}}$ is defined and equals $\tilde{\beta}$ precisely when $\beta_{k,-}^{\hat{J},\hat{r}}$ is defined and equals $\alpha$. Now assume that $j\in\tilde{J}$ for some $1 \leq j \leq N$ and $\sqbinom{\hat{J}}{\hat{r}}=\sqbinom{\tilde{J} \setminus j}{\tilde{r}+1}$ (or equivalently $\hat{r}\geq1$ and $\sqbinom{\tilde{J}}{\tilde{r}}=\sqbinom{\hat{J}+j}{\hat{r}-1}$), and then $\tilde{\alpha}_{k,t}^{\tilde{J},\tilde{r},j}$ is defined and equals $\tilde{\beta}$ if and only if $\beta_{k,b}^{\hat{J},\hat{r},j}$ is defined and equals $\alpha$. In addition, $\tilde{\beta}$ is of the form $\tilde{\alpha}_{k,m}$ if and only if $\beta_{k,d}$ is defined and equals $\alpha$. \label{dualnot}
\end{lem}

\begin{proof}
The verification that the two conditions from Definition \ref{AIdef} are satisfied for $\beta_{k,-}^{\hat{J},\hat{r}}$, $\beta_{k,b}^{\hat{J},\hat{r},j}$, and $\beta_{k,d}$ wherever $\beta \in A_{I+k}$ is easy, establishing the first assertion. Now, from $|\tilde{J}|+\tilde{r}\geq2$ the relation from the second assertion yields $k\in\hat{J}$ and $|\hat{J}|+\hat{r}\geq3$ with our $\sqbinom{\hat{J}}{\hat{r}}$, while if $k\in\hat{J}$ and $|\hat{J}|+\hat{r}\geq3$ then $|\tilde{J}|+\tilde{r}\geq2$ for the value of $\sqbinom{\tilde{J}}{\tilde{r}}$. Moreover, when $\sqbinom{\hat{J}}{\hat{r}}$ and $\sqbinom{\tilde{J}}{\tilde{r}}$ are related in this manner, we interpret the equalities $\tilde{\beta}=\tilde{\alpha}_{k,+}^{\tilde{J},\tilde{r}}$ (i.e., $\mu_{J,r}=m_{+,J,r}^{k,\tilde{J},\tilde{r}}$ for every $J$ and $r$) and $\alpha=\beta_{k,-}^{\hat{J},\hat{r}}$ (namely $m_{J,r}=\mu_{-,J,r}^{k,\hat{J},\hat{r}}$ wherever $|J|+r\geq2$). Both of them are equivalent to the equality $\mu_{J,r}=m_{J,r}$ holding wherever $\sqbinom{J}{r}$ is with $|J|+r\geq2$ and different from both $\sqbinom{\hat{J}}{\hat{r}}$ and $\sqbinom{\tilde{J}}{\tilde{r}}$, as well as to the equivalent equalities $\mu_{\hat{J},\hat{r}}=m_{\hat{J},\hat{r}}+1\geq1$ and $m_{\hat{J},\hat{r}}=\mu_{\hat{J},\hat{r}}-1$ for $\sqbinom{J}{r}=\sqbinom{\hat{J}}{\hat{r}}$ and $\mu_{\tilde{J},\tilde{r}}=m_{\tilde{J},\tilde{r}}-1$, or equivalently $m_{\tilde{J},\tilde{r}}=\mu_{\tilde{J},\tilde{r}}+1\geq1$ (yielding the other admissibility conditions), in case $\sqbinom{J}{r}=\sqbinom{\tilde{J}}{\tilde{r}}$. For the remaining desired equality $\mu_{\emptyset,1}=m_{\emptyset,1}+1$, we note that if there is $\sqbinom{\hat{J}}{\hat{r}}$ with $\mu_{\hat{J},\hat{r}}\geq1$ and $|\hat{J}|+\hat{r}\geq3$ then we must have $\mu_{\emptyset,1}\geq1$ (see the proof of Proposition \ref{genrec}), and that both $\tilde{\beta}$ and $\tilde{\alpha}_{k,+}^{\tilde{J},\tilde{r}}$ are in $\tilde{A}_{I+k}$ when $\tilde{\alpha}\in\tilde{A}_{I}$. This proves the second assertion.

Assuming now that the relation between $\sqbinom{\hat{J}}{\hat{r}}$ and $\sqbinom{\tilde{J}}{\tilde{r}}$ from the third assertion holds for some $j$, and then $\hat{r}\geq1$, $j\in\tilde{J}$, and the conditions $\sqbinom{\tilde{J}}{\tilde{r}}\neq\sqbinom{jk}{0}$ and $\sqbinom{\hat{J}}{\hat{r}}\neq\sqbinom{k}{1}$ are equivalent. In this case we unfold the equalities $\tilde{\beta}=\tilde{\alpha}_{k,t}^{\tilde{J},\tilde{r},j}$ (namely $\mu_{J,r}=m_{t,J,r}^{k,\tilde{J},\tilde{r},j}$ for every $J$ and $r$) and $\alpha=\beta_{k,b}^{\hat{J},\hat{r},j}$ (which means $m_{J,r}=\mu_{b,J,r}^{k,\hat{J},\hat{r},j}$ wherever $|J|+r\geq2$) as follows. Given $\sqbinom{J}{r}$ with $|J|+r\geq2$, both of them mean that $\mu_{J,r}=m_{J,r}$ when $\sqbinom{J}{r}$ does not equal $\sqbinom{\hat{J}}{\hat{r}}$, $\sqbinom{\tilde{J}}{\tilde{r}}$, or $\sqbinom{jk}{0}$; that the equivalent equalities $\mu_{\hat{J},\hat{r}}=m_{\hat{J},\hat{r}}+1\geq1$ and $m_{\hat{J},\hat{r}}=\mu_{\hat{J},\hat{r}}-1$ hold when $\sqbinom{J}{r}=\sqbinom{\hat{J}}{\hat{r}}$; that we have $\mu_{\tilde{J},\tilde{r}}=m_{\tilde{J},\tilde{r}}-1$ i.e., $m_{\tilde{J},\tilde{r}}=\mu_{\tilde{J},\tilde{r}}+1\geq1$, in case $\sqbinom{J}{r}=\sqbinom{\tilde{J}}{\tilde{r}}$; and that $\mu_{jk,0}=m_{jk,0}+1\geq1$, meaning that $m_{jk,0}=\mu_{jk,0}-1$, if $\sqbinom{J}{r}=\sqbinom{jk}{0}$. As all the required admissibility conditions are deduced from the inequalities in the process, and the remaining equality $\mu_{\emptyset,1}=m_{\emptyset,1}$ is now a consequence of the fact that $\tilde{\beta}$ and $\tilde{\alpha}_{t}^{\tilde{J},\tilde{r}}$ are in $\tilde{A}_{I+k}$, the third assertion follows.

Proving the fourth assertion is shorter: The equalities $\tilde{\beta}=\tilde{\alpha}_{k,m}$ and $\alpha=\beta_{k,d}$ both mean that  $\mu_{J,r}=m_{J,r}$ for every vector $\sqbinom{J}{r}$ with $|J|+r\geq2$ that is not $\sqbinom{k}{1}$, and that $\mu_{k,1}=m_{k,1}+1\geq1$ (yielding the admissibility condition) and $m_{k,1}=\mu_{k,1}-1$, and then having both $\tilde{\beta}$ and $\tilde{\alpha}_{k,m}$ in $\tilde{A}_{I+k}$ implies that $\mu_{\emptyset,1}=m_{\emptyset,1}$. This completes the proof of the lemma.
\end{proof}

Recalling the Kronecker $\delta$-symbol, we shall also be using the complementary symbol $\overline{\delta}_{i,j}=1-\delta_{i,j}$, yielding 1 when $i \neq j$ and 0 if $i=j$. We shall also need a truth symbol $\delta_{S}$, which equals 1 when the statement $S$ holds and 0 when it does not. We can now express $y_{I}$ and $y_{I+k}$ via the formula from Corollary \ref{formwden}, and we obtain the following recursive formula for the coefficients $c_{\alpha}$ appearing there, or more precisely $c_{\beta}$ for $\beta \in A_{I+k}$.
\begin{cor}
Take $\beta \in A_{I+k}$ for a multiset $I$ with $|I|\geq2$ and $1 \leq k \leq N$, with multiplicities $\{\mu_{J,r}\}_{|J|+r\geq2}$. Then the coefficient $c_{\beta}$ equals the sum of the expressions \[\textstyle{\sum_{|\hat{J}|+\hat{r}\geq2}\overline{\delta}_{\mu_{\hat{J},\hat{r}},0}\delta_{k\in\hat{J}}\overline{\delta}_{|\hat{J}|+\hat{r},2}(\mu_{\hat{J} \setminus k,\hat{r}}+1)c_{\beta_{k,-}^{\hat{J},\hat{r}}}},\] \[-\textstyle{\sum_{|\hat{J}|+\hat{r}\geq2}\overline{\delta}_{\mu_{\hat{J},\hat{r}},0}\sum_{j=1}^{N}\overline{\delta}_{\mu_{jk,0},0}\overline{\delta}_{\hat{r},0}\overline{\delta}_{\sqbinom{\hat{J}}{\hat{r}},\sqbinom{k}{1}} (\hat{\eta}_{j}+1)(\mu_{\hat{J}+j,\hat{r}-1}+1)c_{\beta_{k,b}^{\hat{J},\hat{r},j}}}\] (where the symbol $\hat{\eta}_{j}$ is the multiplicity with which $j$ appears in the multiset $\hat{J}$), and $-\overline{\delta}_{\mu_{k,1},0}\big(\sum_{|J|+r\geq2}r\mu_{J,r}+\sum_{j=1}^{N}(1+\delta_{j,k})\mu_{jk,0}\big)c_{\beta_{k,d}}$. \label{contcbeta}
\end{cor}

\begin{proof}
Following the proof of Corollary \ref{formwden}, we write $f_{y}^{2n-1}y_{I}$ and $f_{y}^{2n+1}y_{I+k}$, where $n:=|I|$, with the corresponding multiplicities, and Lemma \ref{denfree} shows that the operator $f_{y}^{2}\tfrac{\partial}{\partial x_{j}}-(2n-1)\Delta_{k}f_{y}$ sends the former to the latter. Take the element $\tilde{\beta}\in\tilde{A}_{I+k}$ that is associated with $\beta$, and consider all the contributions involving the corresponding in that action. Lemma \ref{contcalpha} implies that these contributions come from the elements $\tilde{\alpha}\in\tilde{A}_{I}$ (with corresponding element $\alpha \in A_{I}$ as in Lemma \ref{AItildeAI}) such that $\tilde{\beta}$ equals either $\tilde{\alpha}_{k,+}^{\tilde{J},\tilde{r}}$ for an admissible vector $\sqbinom{\tilde{J}}{\tilde{r}}$, or $\tilde{\alpha}_{k,t}^{\tilde{J},\tilde{r},j}$ for appropriate $\sqbinom{\tilde{J}}{\tilde{r}}$ and $j\in\tilde{J}$, or $\tilde{\alpha}_{k,m}$. Lemma \ref{dualnot} implies that this is equivalent to $\tilde{\beta}$ being $\tilde{\alpha}_{k,+}^{\hat{J} \setminus k,\hat{r}}$ where $\alpha$ is the element $\beta_{k,-}^{\hat{J},\hat{r}}$ of $A_{I}$ with $\hat{J}$ being a multiset containing $k$ and $\hat{r}$ being such that $|\hat{J}|+\hat{r}\geq3$ and $\mu_{\hat{J},\hat{r}}\geq1$, or to $\tilde{\beta}$ being $\tilde{\alpha}_{k,t}^{\hat{J}+j,\hat{r}-1}$ for $\alpha=\beta_{k,b}^{\hat{J},\hat{r},j}$ with $\hat{r}\geq1$, $|\hat{J}|+\hat{r}\geq2$, $\sqbinom{\hat{J}}{\hat{r}}\neq\sqbinom{k}{1}$, $\mu_{\hat{J},\hat{r}}\geq1$, and $\mu_{jk,0}\geq1$, or to $\tilde{\beta}=\tilde{\alpha}_{k,m}$ with $\alpha=\beta_{k,d}$ in case $\mu_{k,1}\geq1$.

The value of $c_{\beta}$ is obtained by gathering all of the corresponding contributions from Lemma \ref{contcalpha}, with the appropriate value of $m_{\tilde{J},\tilde{r}}$, and with the $\delta$ and $\overline{\delta}$ symbols presenting the admissibility conditions. This produces the two sums over $|\hat{J}|+\hat{r}\geq2$, and in the term involving $c_{\beta_{k,d}}$, Lemma \ref{AItildeAI} expresses the parameter $m_{\emptyset,1}$ required for completing $\beta_{k,d} \in A_{I}$ (with $|I|=n$) to $\tilde{\beta}_{k,d}$ as $n-1-\sum_{|J|+r\geq2}m_{J,r}$. Definition \ref{AIdef} thus implies that $n-1-m_{\emptyset,1}$ equals $\sum_{|J|+r\geq2}rm_{J,r}+1$, and recalling that for $\beta_{k,d}$ the parameter $m_{J,r}$ was defined to be $\mu_{J,r}-1$ in case $J=\{k\}$ and $r=1$ and $\mu_{J,r}$ in any other case, the coefficient multiplying $c_{\beta_{k,d}}$ is the asserted one. This proves the corollary.
\end{proof}

In order to establish the formula for the coefficients from Corollary \ref{formwden}, we introduce the notation $J!$ for $\prod_{i=1}^{N}\eta_{i}!$ when $J$ contains each $1 \leq i \leq N$ with multiplicity $\eta_{i}$ (complementing the notation $\binom{J}{K}$ used in Definition \ref{DeltaJg}), and define the following combinatorial numbers.
\begin{defn}
Consider, inside the set $A$ from Remark \ref{setpart}, the element $\alpha$ with multiplicities $\{m_{J,r}\}_{\{|J|+r\geq2\}}$. We then define the number \[C_{\alpha}:=\Bigg(\sum_{|J|+r\geq2}rm_{J,r}\Bigg)!\Bigg(\sum_{|J|+r\geq2}m_{J,r}J\Bigg)!\Bigg/\prod_{|J|+r\geq2}r!^{m_{J,r}}m_{J,r}!J!^{m_{J,r}}.\] If $\alpha$ lies in $A_{I,h}$, where $I$ contains each $1 \leq i \leq N$ with multiplicity $\nu_{i}$, then $C_{\alpha}$ counts the number of solutions to the following question: Assume that one is given balls in $N+1$ colors, with the last color being red, such that there are $h-1$ marked red balls, and for each $1 \leq i \leq N$ we have $\nu_{i}$ marked balls. Then $C_{\alpha}$ counts the number of possibilities of putting the balls in $h$ identical boxes, such that for every $J$ and $r$, there are $m_{J,r}$ boxes that contain $r$ red balls and $\eta_{i}$ balls of the $i$th color, where $\eta_{i}$ is the multiplicity to which $J$ contains $i$. \label{combcoeff}
\end{defn}
As in Fa\`{a} di Bruno's formula, as explained in, e.g., \cite{[J1]}, and as in \cite{[Z]}, we have $m_{J,r}!$ in the denominator of $C_{\alpha}$ from Definition \ref{combcoeff} to represent the fact that the boxes are in the combinatorial question given there are identical. The coefficients from Definition \ref{combcoeff} satisfy the following recursive relation.
\begin{prop}
Assume that $\beta=\{\mu_{J,r}\}_{\{|J|+r\geq2\}}$ is an element of $A_{I+k}$ for a set $I$ of size $n\geq2$ and some index $1 \leq k \leq N$. Then $C_{\beta}$ is the sum of \[\textstyle{\sum_{|\hat{J}|+\hat{r}\geq2}\overline{\delta}_{\mu_{\hat{J},\hat{r}},0}\delta_{k\in\hat{J}}\overline{\delta}_{|\hat{J}|+\hat{r},2}(\mu_{\hat{J} \setminus k,\hat{r}}+1)C_{\beta_{k,-}^{\hat{J},\hat{r}}}},\] \[\textstyle{\sum_{|\hat{J}|+\hat{r}\geq2}\overline{\delta}_{\mu_{\hat{J},\hat{r}},0}\sum_{j=1}^{N}\overline{\delta}_{\mu_{jk,0},0}\overline{\delta}_{\hat{r},0}\overline{\delta}_{\sqbinom{\hat{J}}{\hat{r}},\sqbinom{k}{1}} (\hat{\eta}_{j}+1)(\mu_{\hat{J}+j,\hat{r}-1}+1)C_{\beta_{k,b}^{\hat{J},\hat{r},j}}}\] (where again $\hat{\eta}_{j}$ denotes the multiplicity of $j$ in $\hat{J}$), and the two expressions $\overline{\delta}_{\mu_{k,1},0}\sum_{j=1}^{N}(1+\delta_{j,k})\mu_{jk,0}C_{\beta_{k,d}}$ and $\overline{\delta}_{\mu_{k,1},0}\sum_{|J|+r\geq2}r\mu_{J,r}C_{\beta_{k,d}}$. \label{indcomb}
\end{prop}
As in \cite{[Z]}, the two multiples of $C_{\beta_{k,d}}$ will play different roles in the proofs of Proposition \ref{indcomb}, both the one immediately following and the combinatorial one given after Theorem \ref{coeffval} below.

\begin{proof}
By setting $h:=\sum_{|J|+r\geq2}\mu_{J,r}$, we find that $\beta \in A_{I+k,h}$, and all the asserted terms involve, by Lemma \ref{dualnot}, coefficients $C_{\alpha}$ for elements $\alpha \in A_{I}$. Assume first that $\hat{J}$ is a multiset containing $k$ and that $\hat{r}$ is a number for which we have $|\hat{J}|+\hat{r}\geq3$ and $\mu_{\hat{J},\hat{r}}\geq1$, meaning that the $\overline{\delta}$-symbols do not vanish for the corresponding summand in the first asserted sum. The fact that in the definition of $\beta_{k,-}^{\hat{J},\hat{r}}$ we subtract 1 from $\mu_{\hat{J},\hat{r}}$, add 1 to $\mu_{\hat{J} \setminus k,\hat{r}}$, and change nothing more (compared to $\beta$) means that the sum $\sum_{|J|+r\geq2}\mu_{J,r}$, or equivalently $\sum_{|J|+r\geq2}r\mu_{J,r}$, remains unchanged, and thus $\beta_{k,-}^{\hat{J},\hat{r}} \in A_{I,h}$. The contribution of the corresponding summand is therefore $\mu_{\hat{J} \setminus k,\hat{r}}+1$ times $(h-1)!I!\big/\prod_{|J|+r\geq2}\mu_{-,J,r}^{k,\hat{J},\hat{r}}!(r!J!)^{\mu_{-,J,r}^{k,\hat{J},\hat{r}}}$. Substituting the $\mu_{-,J,r}^{k,\hat{J},\hat{r}}$'s, we get that the total power of each $r!$ in the denominator, as well as that of the $\eta_{i}$'s for $i \neq k$ that form the expression $J!$ (which are the same for $\hat{J}$ and $\hat{J} \setminus k$), are the same as for $C_{\beta}$, and so are the other parts of the denominator that arise from vectors $\sqbinom{J}{r}$ that do not equal $\sqbinom{\hat{J}}{\hat{r}}$ or $\sqbinom{\hat{J} \setminus k}{\hat{r}}$. Combining the remaining parts of the denominator, in which $\hat{\eta}_{k}$ is the multiplicity of $k$ in $\hat{J}$ (and thus the multiplicity in $\hat{J} \setminus k$ is $\hat{\eta}_{k}-1$), with the external multiplier yields \[\frac{\mu_{\hat{J} \setminus k,\hat{r}}+1}{(\hat{\eta}_{k}\!-\!1\!)!^{\mu_{\hat{J} \setminus k,\hat{r}}+1}\!(\mu_{\hat{J} \setminus k,\hat{r}}\!+\!1\!)!\hat{\eta}_{k}!^{\mu_{\hat{J},\hat{r}}-1}\!(\mu_{\hat{J},\hat{r}}\!-\!1\!)!}\!=\!\frac{\hat{\eta}_{k}\mu_{\hat{J},\hat{r}}}{(\hat{\eta}_{k}\!-\!1\!)!^{\mu_{\hat{J} \setminus k,\hat{r}}}\mu_{\hat{J} \setminus k,\hat{r}}!\hat{\eta}_{k}!^{\mu_{\hat{J},\hat{r}}}\mu_{\hat{J},\hat{r}}!},\] meaning that for such $\hat{J}$ and $\hat{r}$ the total expression equals $\frac{\hat{\eta}_{k}\mu_{\hat{J},\hat{r}}}{\nu_{k}+1}C_{\beta}$ (the denominator here arising from the numerator in Definition \ref{combcoeff} for $C_{\beta}$ containing $(\nu_{k}+1)!$ instead of our $\nu_{k}!$). Noting that the vanishing of $\overline{\delta}_{\mu_{\hat{J},\hat{r}},0}\delta_{k\in\hat{J}}$ is equivalent to that of the numerator $\hat{\eta}_{k}\mu_{\hat{J},\hat{r}}$, only the restriction $|\hat{J}|+\hat{r}\geq3$ from the multiplier $\overline{\delta}_{|\hat{J}|+\hat{r},2}$ remains meaningful. It follows that these terms contribute $\frac{C_{\beta}}{\nu_{k}+1}\sum_{|\hat{J}|+\hat{r}\geq3}\hat{\eta}_{k}\mu_{\hat{J},\hat{r}}$ in total.

We now take some $1 \leq j \leq N$ with $\mu_{jk,0}\geq1$, and consider $\hat{r}\geq1$ and a multiset $\hat{J}$, with $\sqbinom{\hat{J}}{\hat{r}}\neq\sqbinom{k}{1}$, for which $|\hat{J}|+\hat{r}\geq2$ and $\mu_{\hat{J},\hat{r}}\geq1$ (so that the corresponding $\overline{\delta}$-symbols do not vanish), and write $\hat{\eta}_{j}$ for the multiplicity of $j$ in $\hat{J}$ (so that its multiplicity in $\hat{J}+j$ is $\hat{\eta}_{j}+1$). We have $h\geq2$ under these conditions, and since for $\beta_{k,b}^{\hat{J},\hat{r},j}$ we subtract 1 from $\mu_{jk,0}$ and $\mu_{\hat{J},\hat{r}}$ and add 1 to $\mu_{\hat{J}+j,\hat{r}-1}$, we obtain that this element lies in $A_{I,h-1}$. The corresponding contribution is $(\hat{\eta}_{j}+1)(\mu_{\hat{J}+j,\hat{r}-1}+1)$ times $(h-2)!I!\big/\prod_{|J|+r\geq2}\mu_{b,J,r}^{k,\hat{J},\hat{r},j}!(r!J!)^{\mu_{b,J,r}^{k,\hat{J},\hat{r},j}}$. The denominators arising from any vector $\sqbinom{J}{r}$ other than $\sqbinom{\hat{J}}{\hat{r}}$, $\sqbinom{\hat{J}+j}{\hat{r}-1}$, or $\sqbinom{jk}{0}$ are the same ones as in $C_{\beta}$, and so are the powers of $\eta_{i}!$ for $i \neq j$ from $J!$ in the first two among the remaining vectors. The external multiplier and the remaining parts of the denominator give \[\frac{(\hat{\eta}_{j}+1)(\mu_{\hat{J}+j,\hat{r}-1}+1)\big/(1+\delta_{j,k})^{\mu_{jk,0}-1}(\mu_{jk,0}-1)!}{[(\hat{\eta}_{j}+1)!(\hat{r}-1)!]^{\mu_{\hat{J}+j,\hat{r}-1}+1}(\mu_{\hat{J}+j,\hat{r}-1}+1)! [\hat{\eta}_{j}!\hat{r}!]^{\mu_{\hat{J},\hat{r}}-1}(\mu_{\hat{J},\hat{r}}-1)!}\] in total. This amounts to $(1+\delta_{j,k})\mu_{jk,0}\hat{r}\mu_{\hat{J},\hat{r}}$ over the parts of the denominator of $C_{\beta}$ that are associated with our three vectors, and recalling the numerator of $C_{\beta}$, the contribution of such an element is $\frac{(1+\delta_{j,k})\mu_{jk,0}\hat{r}\mu_{\hat{J},\hat{r}}}{(\nu_{k}+1)(h-1)}C_{\beta}$. Summing over $\hat{J}$, $\hat{r}$, and $j$, and observing that the vanishing of $\overline{\delta}_{\mu_{\hat{J},\hat{r}},0}\overline{\delta}_{\mu_{jk,0},0}\overline{\delta}_{\hat{r},0}$ is implied by the corresponding vanishing of the numerator (but not the one of the $\overline{\delta}$ associated with the restriction $\sqbinom{\hat{J}}{\hat{r}}\neq\sqbinom{k}{1}$), the sum over these elements equals $\sum_{j=1}^{N}\frac{(1+\delta_{j,k})\mu_{jk,0}C_{\beta}}{(\nu_{k}+1)(h-1)}\sum_{\sqbinom{\hat{J}}{\hat{r}}\neq\sqbinom{k}{1}}\hat{r}\mu_{\hat{J},\hat{r}}$.

It remains to consider, when $\mu_{k,1}\geq1$ (and thus $h\geq2$ once more), the two terms involving $C_{\beta_{k,d}}$, which lies in $A_{I,h-1}$ due to the subtraction of 1 from $\mu_{k,1}$. These terms are $\sum_{j=1}^{N}(1+\delta_{j,k})\mu_{jk,0}$ and $\sum_{|J|+r\geq2}r\mu_{J,r}=h-1$ times $(h-2)!I!\big/\prod_{|J|+r\geq2}\mu_{d,J,r}^{k}!(r!J!)^{\mu_{d,J,r}^{k}}$. All the denominators other than the one arising from $\sqbinom{J}{r}\neq\sqbinom{k}{1}$ coincide with those appearing in $C_{\beta}$, and as the remaining one is the multiplier $\frac{1}{(\mu_{k,1}-1)!}=\frac{\mu_{k,1}}{\mu_{k,1}!}$ for $C_{\beta_{k,d}}$ and $\frac{1}{\mu_{k,1}!}$ in $C_{\beta}$, the same numerator considerations identify our terms as $\sum_{j=1}^{N}\frac{(1+\delta_{j,k})\mu_{jk,0}\mu_{k,1}}{(\nu_{k}+1)(h-1)}C_{\beta}$ and $\frac{\mu_{k,1}C_{\beta}}{\nu_{k}+1}$, with the symbol $\overline{\delta}_{\mu_{k,1},0}$ becoming redundant because of the multiplier $\mu_{k,1}$.

We therefore have to consider $\frac{C_{\beta}}{\nu_{k}+1}\sum_{|\hat{J}|+\hat{r}\geq3}\hat{\eta}_{k}\mu_{\hat{J},\hat{r}}$ plus \[\sum_{j=1}^{N}\frac{(1+\delta_{j,k})\mu_{jk,0}C_{\beta}}{(\nu_{k}+1)(h-1)}\sum_{\sqbinom{\hat{J}}{\hat{r}}\neq\sqbinom{k}{1}}\hat{r}\mu_{\hat{J},\hat{r}}+\sum_{j=1}^{N}\frac{(1+\delta_{j,k})\mu_{jk,0}\mu_{k,1}} {(\nu_{k}+1)(h-1)}C_{\beta}+\frac{\mu_{k,1}C_{\beta}}{\nu_{k}+1},\] involving four terms. Noting that $\mu_{k,1}$ in the third term is just the summand $\hat{r}\mu_{\hat{J},\hat{r}}$ for the missing vector $\sqbinom{\hat{J}}{\hat{r}}=\sqbinom{k}{1}$ in the second term, and together the sum $\sum_{|\hat{J}|+\hat{r}\geq2}\hat{r}\mu_{\hat{J},\hat{r}}$ equals $h-1$ via Definition \ref{AIdef} and Remark \ref{setpart}, the $j$th summand in the sum of those two terms equals just $\frac{(1+\delta_{j,k})\mu_{jk,0}C_{\beta}}{\nu_{k}+1}$. But the multiplier of $\frac{C_{\beta}}{\nu_{k}+1}$ here is $\hat{\eta}_{k}\mu_{\hat{J},\hat{r}}$ for $\sqbinom{\hat{J}}{\hat{r}}=\sqbinom{jk}{0}$, the multiplier $\mu_{k,1}$ in the fourth term is the one associated with $\sqbinom{\hat{J}}{\hat{r}}=\sqbinom{k}{1}$, and for any other vector $\sqbinom{\hat{J}}{\hat{r}}$ with $|\hat{J}|+\hat{r}=2$ we have $\hat{\eta}_{k}=0$ and therefore no contribution. It means that the total expression is $\frac{C_{\beta}}{\nu_{k}+1}\sum_{|\hat{J}|+\hat{r}\geq2}\hat{\eta}_{k}\mu_{\hat{J},\hat{r}}$, and the sum equals $\nu_{k}+1$, thus canceling the denominator, for $\beta \in A_{I+k}$ by Definition \ref{AIdef}. This proves the proposition.
\end{proof}

We can now prove the following explicit formula.
\begin{thm}
For every element $\alpha$ in the set $A_{I}$ from Definition \ref{AIdef}, where $I$ is the multiset containing each $1 \leq i \leq N$ with multiplicity $\nu_{i}$, of size $\sum_{i=1}^{N}\nu_{i}\geq2$, let $C_{\alpha}$ be the constant from Definition \ref{combcoeff}. Then the expression $\prod_{|J|+r\geq2}(\Delta_{J}f_{y^{r}})^{m_{J,r}}\Big/f_{y}^{n+\sum_{|J|+r\geq2}m_{J,r}}$ corresponding to $\alpha$ in the expression for $y_{I}$ in Corollary \ref{formwden} comes multiplied by the coefficient $(-1)^{\sum_{|J|+r\geq2}m_{J,r}}C_{\alpha}$, i.e., we have the formula \[y_{I}=\sum_{\alpha \in A_{I}}\frac{(-1)^{\sum_{|J|+r\geq2}m_{J,r}}\big(\sum_{|J|+r\geq2}rm_{J,r}\big)!I!}{\prod_{|J|+r\geq2}r!^{m_{J,r}}m_{J,r}!J!^{m_{J,r}}}\cdot \frac{\prod_{|J|+r\geq2}(\Delta_{J}f_{y^{r}})^{m_{J,r}}}{f_{y}^{n+\sum_{|J|+r\geq2}m_{J,r}}}.\] \label{coeffval}
\end{thm}

\begin{proof}
Having the formula from Corollary \ref{formwden} at hand, we only need to prove the equality $c_{\alpha}=(-1)^{\sum_{|J|+r\geq2}m_{J,r}}C_{\alpha}$ for every $\alpha \in A_{I}$. We saw that when $|I|=2$, the set $A_{I}$ consists of a unique element, in $A_{I,1}$, in which $m_{J,r}$ equals 1 when $J=I$ and $r=0$ and 0 otherwise, with $C_{\alpha}=1$ in Definition \ref{combcoeff}, and the formula from Equation \eqref{yder2} yields the desired result. We therefore assume, by induction, that the statement holds for every $\alpha \in A_{I}$, take some $1 \leq k \leq N$, and wish to prove the formula for an element $\beta=\{\mu_{J,r}\}_{\{|J|+r\geq2\}}$ of $A_{I+k}$ (this will clearly establish the result by induction for every multiset $I$). Corollary \ref{contcbeta} expresses $c_{\beta}$ as a combination of coefficients $c_{\alpha}$ for $\alpha \in A_{I}$, and note that if $\beta \in A_{I+k,h}$ for some $h$ then if $\beta_{k,-}^{\hat{J},\hat{r}}$ is defined then it is in $A_{I,h}$ and the corresponding contribution from that corollary is with a $+$ sign, but when an element $\beta_{k,b}^{\hat{J},\hat{r},j}$ or $\beta_{k,d}$ is defined, it belongs to $A_{I,h-1}$ and yields a contribution with a $-$ sign. The induction hypothesis allows us to write $c_{\alpha}$ as $(-1)^{h}C_{\alpha}$ or $(-1)^{h-1}C_{\alpha}$ (according to whether our $\alpha$ is in $A_{I,h}$ or in $A_{I,h-1}$), with $C_{\alpha}$ from Definition \ref{combcoeff}, and then $(-1)^{h}c_{\beta}$ becomes the formula from Proposition \ref{indcomb}. As this proposition compares that expression with $C_{\beta}$, we obtain the desired formula. This proves the theorem.
\end{proof}
Note that the formula from Theorem \ref{coeffval} is indeed symmetric, in the sense that if we apply a permutation to the indices $1 \leq i \leq N$ that keeps $I$ invariant, the expression for $y_{I}$ remains invariant under this operation (as it should be).

\smallskip

For examining Theorem \ref{coeffval}, let us consider the formulae that we have for $I$ of size $n=3$ or $n=4$. In the former case the explicit formula is given in Equation \eqref{yder3fin}, and we saw that in this case $A_{I,1}$ consists of a single element in which $m_{I,0}=1$ and the other multiplicities vanishing, while $A_{I,2}$ contains one element with $m_{i,1}=m_{I \setminus i,0}=1$ and the other $m_{J,r}$'s vanishing, for every $i \in I$. Now, Definition \ref{combcoeff} assigns the value 1 to the element of $A_{I,1}$, and one can verify that it associates 1 to an element of $A_{I,2}$ with $i \not\in I \setminus i$, 2 when $i \in I \setminus i$ but $I\neq\{i,i,i\}$, and 3 in case $I=\{i,i,i\}$. As this coincides with the number of summands in Equation \eqref{yder3fin} that contribute 1 to $\Delta_{i}f_{y}\cdot\Delta_{I \setminus i}f$, this equation is indeed in correspondence with Theorem \ref{coeffval}. When $n=4$ and $I$ is a set, with no multiplicities, the set $A_{I}$ was seen to consist of the following elements: $A_{I,1}$ has only the element with $m_{I,0}=1$ (every multiplicity without a written value is meant to vanish); In $A_{I,2}$ there are the 4 elements with $m_{i,1}=m_{I \setminus i,0}=1$ and the 6 elements in which $m_{J,1}=m_{I \setminus J,0}=1$ for a subset $J \subseteq I$ of size 2; And $A_{I,3}$ contains 3 elements having $m_{J_{1},0}=m_{J_{2},0}=m_{\emptyset,2}=1$ for unordered partition of $I$ into two sets of size 2, as well as 6 elements with $m_{J,0}=1$ and $m_{i,1}=1$ for the two elements $i \in I \setminus J$. Now, in Definition \ref{combcoeff}, the $\nu_{i}$'s from $I!$, the $\eta_{i}$'s from $J!$, and the $m_{J,r}$'s are all 0 or 1, so that the coefficient $C_{\alpha}$ is just $(h-1)!/\prod_{J,r}r!^{m_{J,r}}$, and since this gives 2 for the second type of elements in $A_{I,3}$ and 1 for all the rest, Equation \eqref{yder4} is indeed the incarnation of Theorem \ref{coeffval} for this case (one can verify that when some of the elements of a multiset $I$ of size 4 coincide, the identification of elements in Equation \eqref{yder4} and the resulting coefficients in Definition \ref{combcoeff} indeed match one another for such a case as well, as Remark \ref{Hardy} below predicts).

\medskip

The coefficients $C_{\alpha}$ from Definition \ref{combcoeff} have a combinatorial interpretation, meaning that Proposition \ref{indcomb} (and thus also Theorem \ref{coeffval}) should have a combinatorial proof as well. We now present this proof, which parallels the one we already gave, but from the combinatorial viewpoint. Take $\beta \in A_{I+k,h}$, which we again write as $\{\mu_{J,r}\}_{\{|J|+r\geq2\}}$, where the multiplicity of $i$ in $I$ is $\nu_{i}$. Recall that $C_{\beta}$ counts the number of ways to put a collection of marked balls, with $\nu_{i}+\delta_{i,k}$ of the $i$th color and $h-1$ red ones, into $h$ identical boxes, under the restriction that if $J$ is a multiset containing each index $i$ with multiplicity $\eta_{i}$ then the number of boxes containing $\eta_{i}$ balls of the $i$th color and $r$ red ones is $\mu_{J,r}$. We investigate this set of possibilities according to solutions of a similar problem with the elements of $A_{I}$ that show up in Proposition \ref{indcomb}, recalling that we only allow boxes containing at least two balls (of any mixture of colors).

Consider the ball of the $k$th color, with the marking $\nu_{k}+1$. In some possibilities it will lie in a box containing at least 3 balls, i.e., the content of the corresponding box is described by $\binom{\hat{J}}{\hat{r}}$, with $k\in\hat{J}$ and with $|\hat{J}|+\hat{r}\geq3$. Taking the ball $\nu_{k}+1$ out yields a solution to the problem of possibilities for $\beta_{k,-}^{\hat{J},\hat{r}}$. Conversely, for every solution for $\beta_{k,-}^{\hat{J},\hat{r}}$ we can get a possibility for $\beta$ by adding the ball $\nu_{k}+1$ to any of the boxes with ball content $\binom{\hat{J} \setminus k}{\hat{r}}$, and the number of such boxes is $\mu_{-,\hat{J} \setminus k,\hat{r}}^{k,\hat{J},\hat{r}}=\mu_{\hat{J} \setminus k,\hat{r}}+1$. The first term in Proposition \ref{indcomb} therefore counts all the possibilities for $\beta$ where the ball $\nu_{k}+1$ sits with at least two other balls in its box.

The next possibilities that we focus on are those where the ball $\nu_{k}+1$, of the $k$th color, shares its box with a single other ball, which is red. Then $\mu_{k,1}\geq1$ and thus $h\geq2$, and the red ball can have any of the $h-1$ markings. Removing the entire box produces a solution for the question associated with $\beta_{k,d}$ (with the markings of the red balls being the complement of the one in the box we took out), and for any of the $h-1$ possible red markings, and any such solution for $\beta_{k,d}$ with the chosen red marking missing, we can add a box with the chosen missing red ball and our ball $\nu_{k}+1$, and get a possibility for $\beta$. This means that the fourth term in Proposition \ref{indcomb} (where we recall that the sum in this term equals $h-1$) involves precisely those possibilities for $\beta$ in which $\nu_{k}+1$ is one of two balls in a box, and the other ball in that box is red.

It remains to count the possibilities in which the ball $\nu_{k}+1$ has only one box-mate, and this box-mate is not red, but rather of the $j$th color. In this case $\mu_{jk,0}\geq1$, so that $h\geq2$ once again because $\beta \in A_{I+k,h}$ and $n=|I|\geq2$. Taking out the box will not give a solution for an element of $A_{I}$, so that in order to obtain such a solution, we remove our ball $\nu_{k}+1$ and the red ball $h-1$ (and the box), and put the box-mate of $\nu_{k}+1$ in the place of the red ball $h-1$. The box containing the red ball $h-1$ is associated with $\hat{J}$ and $\hat{r}$ such that $|\hat{J}|+\hat{r}\geq2$ (as always) and $\hat{r}\geq1$ (due to the existence of the red ball $h-1$), and as the box-mate of $\nu_{k}+1$ is of the $j$th color, after this operation the box of $h-1$ will now be associated with $\sqbinom{\hat{J}+j}{\hat{r}-1}$. Moreover, the removal of the box itself means that we took out a box corresponding to $\sqbinom{jk}{0}$, so that after this process we get a solution for the problem corresponding to $\beta_{k,b}^{\hat{J},\hat{r},j}$ when $\sqbinom{\hat{J}}{\hat{r}}\neq\sqbinom{k}{1}$, and to $\beta_{k,d}$ in case $\sqbinom{\hat{J}}{\hat{r}}=\sqbinom{k}{1}$. On the other hand, from $\hat{J}$ and $\hat{r}$ with $\sqbinom{\hat{J}}{\hat{r}}\neq\sqbinom{k}{1}$ and a solution to the question for $\beta_{k,b}^{\hat{J},\hat{r},j}$, for producing a possibility for $\beta$ we choose a box corresponding to $\sqbinom{\hat{J}+j}{\hat{r}-1}$ (of which there are $\mu_{b,\hat{J}+j,\hat{r}+1}^{k,\hat{J},\hat{r},j}=\mu_{\hat{J}+j,\hat{r}+1}+1$), and a ball of color $j$ in it (and there are $\hat{\eta}_{j}+1$ of those, when the set is $\hat{J}+j$), and we replace this ball by the red ball $h-1$ and take an extra box to put that ball and our ball $\nu_{k}+1$ in it. Similarly, with solution for the problem for $\beta_{k,d}$ we do the same for $\hat{J}=\{k\}$ and $\hat{r}=1$, so that we take a box with a ball of color $k$ and a ball of color $j$ (there are $\mu_{d,jk,0}^{k}=\mu_{jk,0}$ of those), take a ball of color $j$ in it (of which there are $1+\delta_{j,k}$), and carry out the same process of replacement and an extra box. It means that the $j$th summand in the third term in Proposition \ref{indcomb} counts the possibilities for $\beta$ in which our ball $\nu_{k}+1$ is in a box with only one ball, of the color $j$, and the red ball $h-1$ is in a box with a single box-mate of color $k$, and the $j$th summand in the second term in that proposition accounts for the possibilities for $\beta$ where $\nu_{k}+1$ still has a single box-mate of color $j$, and the red ball $h-1$ sits in any other type of box.

Thus, for understanding the combinatorial meaning of the proof of Proposition \ref{indcomb}, we note that in the possibilities for $\beta$, the ball of color $k$ and marking $\nu_{k}+1$ can lie either in a box associated with $\sqbinom{\hat{J}}{\hat{r}}$ for which $k\in\hat{J}$ and $|\hat{J}|+\hat{r}\geq3$ (assuming that $\mu_{\hat{J},\hat{r}}\geq1$), or in a box of type $\sqbinom{k}{1}$ if $\mu_{k,1}\geq1$, or in a box corresponding to $\sqbinom{jk}{0}$ for some $j$ when $\mu_{jk,0}\geq1$ (this accounts for all the possible pairs $\sqbinom{\hat{J}}{\hat{r}}$ with $|\hat{J}|+\hat{r}\geq2$ and $k\in\hat{J}$). When the third option occurs we have $h\geq2$, and the red ball $h-1$ can either belong to a box associated with $\sqbinom{k}{1}$ (in case $\mu_{k,1}\geq1$ again), or in any other type $\sqbinom{\hat{J}}{\hat{r}}\neq\sqbinom{k}{1}$ of box, provided that $\mu_{\hat{J},\hat{r}}\geq1$, $|\hat{J}|+\hat{r}\geq2$, and $\hat{r}\geq1$ (as the box contains a red ball). As it is clear that these options are mutually exclusive, and we saw that each of these options produces a number of possibilities that is the appropriate multiple of the number of solutions for the question associated with the corresponding $\alpha \in A_{I}$ (this $\alpha$ was $\beta_{k,-}^{\hat{J},\hat{r}}$, $\beta_{k,d}$, $\beta_{k,d}$ once more, and $\beta_{k,b}^{\hat{J},\hat{r},j}$ respectively), this completes the combinatorial proof of Proposition \ref{indcomb}, and consequently of Theorem \ref{coeffval} as well.

\section{The Meaning of the Expressions $\Delta_{J}f_{y^{r}}$ \label{Expansion}}

The expression for $y_{I}$ in Theorem \ref{coeffval} simplifies under the assumption that $f_{i}=0$ for every $1 \leq i \leq N$. Then, in the sum defining $\Delta_{J}g$ in Definition \ref{DeltaJg}, all the summands in which $K$ is not the empty multiset vanish, so that $\Delta_{J}g$ reduces to $g_{J}f_{y}^{|J|}$, and in particular we have $\Delta_{J}f_{y^{r}}=f_{Jy^{r}}f_{y}^{|J|}$ for every $J$ and $r$. Substituting these into Theorem \ref{coeffval}, and cancelling the product over all $J$ and $r$ of the powers $f_{y}^{|J|m_{J,r}}$ with $f_{y}^{n}$ from the numerator (since $\sum_{J,r}|J|m_{J,r}=|I|=n$), yields the formula
\begin{equation}
y_{I}=\sum_{\alpha \in A_{I}}\frac{(-1)^{\sum_{|J|+r\geq2}m_{J,r}}\big(\sum_{|J|+r\geq2}rm_{J,r}\big)!I!}{\prod_{|J|+r\geq2}r!^{m_{J,r}}m_{J,r}!J!^{m_{J,r}}}\cdot\frac{\prod_{|J|+r\geq2}f_{Jy^{r}}^{m_{J,r}}}{f_{y}^{\sum_{|J|+r\geq2}m_{J,r}}}. \label{fi=0}
\end{equation}
As I am not aware of a publication containing the partial derivative analogue of Theorem 2 of \cite{[J3]} or of Equation (7) of \cite{[Wi]} in this generality (Equation (2.9) of \cite{[Y]} is close to such a formula, but under some restrictions on the value of $f_{y}$), it seems that Equation \eqref{fi=0}, in this setting is also new. It will also be a special case of Theorem \ref{claselts} below, which can be proved in a manner similar to that from \cite{[J3]} or from \cite{[Wi]}, with our notation. We will prove it, however, using the adaptation of the argument from \cite{[Z]}.

To do so, let $\{f_{i}(\vec{x}_{0},y_{0})\}_{i=1}^{N}$ be arbitrary yet again, and define the function $\varphi$, also of $N+1$ variables, by setting $\varphi(\vec{x},z):=f(\vec{x},z+\vec{\lambda}\cdot\vec{x})$ for a vector $\vec{x}$ and another variable $z$, where $\vec{\lambda}$ is a vector $(\lambda_{1},\ldots,\lambda_{N})\in\mathbb{R}^{N}$ and $\vec{\lambda}\cdot\vec{x}$ is the standard scalar product $\sum_{i=1}^{N}\lambda_{i}x_{i}$. For $z_{0}=y_{0}-\vec{\lambda}\cdot\vec{x}_{0}$ the derivative of $\varphi$ with respect to $z$ yields $\varphi_{z}(\vec{x}_{0},y_{0}-\vec{\lambda}\cdot\vec{x}_{0})=f_{y}(\vec{x}_{0},y_{0})\neq0$, meaning that $z$ can be given as a function of $\vec{x}$ in a neighborhood of $\vec{x}_{0}$ by the equation $\varphi(\vec{x},z)=0$ (and the initial condition $z(\vec{x}_{0})=z_{0}=y_{0}-\vec{\lambda}\cdot\vec{x}_{0}$). It is now clear that with $z$ thus defined and $y$ given in terms of $f(\vec{x},y)=0$ we have $y(x)=z(x)+\vec{\lambda}\cdot(\vec{x}-\vec{x}_{0})$. Differentiating with respect to $x_{i}$ produces the equality $\varphi_{i}=f_{i}+\lambda_{i}f_{y}$ for every $1 \leq i \leq N$, and therefore if we consider the derivatives of $f$ at $(\vec{x}_{0},y_{0})$ and choose each $\lambda_{i}$ to be $y_{i}(\vec{x}_{0})=-\frac{f_{i}}{f_{y}}$, then we obtain $z_{i}(\vec{x}_{0})=0$ for every $i$. We can now prove the following relation.
\begin{lem}
For this choice of $\vec{\lambda}$ we get $\varphi_{Jz^{r}}=\Delta_{J}f_{y^{r}}/f_{y}^{|J|}$ for every multiset $J$ and integer $r\geq0$. \label{trans}
\end{lem}

\begin{proof}
It is clear (e.g., by induction on $r$) that $\varphi_{z^{r}}(\vec{x},z)=f_{y^{r}}(\vec{x},y)$ under the relation $y=z+\vec{\lambda}\cdot(\vec{x}-\vec{x}_{0})$. When applying a derivative with respect to some $x_{i}$, we get $\varphi_{iz^{r}}(\vec{x},z)=f_{iy^{r}}(\vec{x},y)+\lambda_{i}f_{y^{r+1}}(\vec{x},y)$. By doing it repeatedly, with the various indices $1 \leq i \leq N$, according to the respective multiplicities $\eta_{i}$, and using the binomial identity in the part of $\binom{J}{K}$ that corresponds to the index of differentiation, we obtain by induction that $\varphi_{Jz^{r}}$ equals $\sum_{K \subseteq J}(-1)^{|K|}\binom{J}{K}f_{(J \setminus K)y^{r+|K|}}\cdot\prod_{i \in K}\lambda_{i}$. Recalling the value $-\frac{f_{i}}{f_{y}}$ of $\lambda_{i}$, this is indeed obtained from the sum from Definition \ref{DeltaJg}, in which we take $g$ to be $f_{y^{r}}$, after division by $f_{y}^{|J|}$. This proves the lemma.
\end{proof}
Using Lemma \ref{trans}, knowing that the formula from Equation \eqref{fi=0} (in which elementary derivatives of $f$ appear, and not the heavier combinations from Definition \ref{DeltaJg}) is valid when the first derivatives $y_{i}$, $1 \leq i \leq N$ of $y$ vanish at $\vec{x}_{0}$, yields the general formula given in Theorem \ref{coeffval}. To see this, we define $z$ in terms of $\varphi$ as above, with the values of the $\lambda_{i}$'s with which $z_{i}(\vec{x}_{0})=0$ for every $i$, and then Equation \eqref{fi=0} gives the expression for $z_{I}$ using the derivatives of $\varphi$ with respect to the variables $x_{i}$, $1 \leq i \leq N$, and $z$. The fact that the difference between $y$ and $z$ is a linear function of $\vec{x}$ implies that $y_{I}$ is the same as $z_{I}$ when $|I|\geq2$, but we need it in terms of the derivatives of $f$ and not $\varphi$. Lemma \ref{trans} does the required transformation (including replacing every $\varphi_{z}$ in the denominator by $f_{y}$), and as the total denominator will also contain $f_{y}$ raised to the power $\sum_{J,r}|J|m_{J,r}=|I|=n$, the the formula from Theorem \ref{coeffval} follows.

Theorem 17 of \cite{[Z]} reproduces, in the case of an implicit function of one variable, the formula for $y^{(n)}$ from \cite{[J3]} and \cite{[Wi]}, which contains just the products of partial derivatives of $f$, without the vanishing assumption. We now aim to state and prove this formula in our setting. Thus, in analogy to Definition \ref{AIdef}, we let $B_{I}$ be the set of multiplicities $\{s_{H,t}\}_{H,t}$, taken over multisets $H$ and integers $t\geq0$, such that the equality $\sum_{H,t}s_{H,t}H=I$ is satisfied as multisets and we have $\sum_{H,t}(t-1)s_{H,t}=-1$ as numbers, as well as $s_{\emptyset,0}=s_{\emptyset,1}=0$ (adding $s_{i,0}=0$ for every $1 \leq i \leq N$ to elements of $A_{I}$ shows that $A_{I} \subseteq B_{I}$, but the inclusion is typically strict, as $s_{i,0}$ are allowed to be positive in elements of $B_{I}$). The same idea from Remark \ref{setpart} partitions $B_{I}$ as $\bigcup_{g}B_{I,g}$, where the index $g$ for which an element $\{s_{H,t}\}_{H,t} \in B_{I}$ lies in $B_{I,g}$ is $\sum_{H,t}s_{H,t}$. Once again we immediately have $g\geq1$, and note that if $|I|=n$ then $\sum_{|H|+t\geq2}(|H|+t)s_{H,t}$ is bounded from below by $2\sum_{|H|+t\geq2}s_{H,t}=2\big(g-\sum_{i=1}^{N}s_{i,0}\big)$ but equals $n+g-1-\sum_{i=1}^{N}s_{i,0}$, and as $\sum_{i=1}^{N}s_{i,0}\leq\sum_{H,t}|H|s_{H,t}=|I|=n$, we obtain that the union goes over $1 \leq g\leq2n-1$.

We also set $B$ to be the union $\bigcup_{|I|\geq1}B_{I}$ (a disjoint union, since $I$ is determined as the multiset sum $\sum_{H,t}s_{H,t}H$ of an element of $B$), where for a singleton $I=\{i\}$ the set $B_{i}$ is not empty, but rather consisting of the single element in which $s_{i,0}=1$ and the other multiplicities vanish (but $B_{\emptyset}$ is empty). Analogously to Definition \ref{combcoeff}, we set
\begin{equation}
D_{\gamma}:=\Bigg(\sum_{H,t}ts_{H,t}\Bigg)!\Bigg(\sum_{H,t}s_{H,t}H\Bigg)!\Bigg/\prod_{H,t}t!^{s_{H,t}}s_{H,t}!H!^{s_{H,t}} \label{Dgamma}
\end{equation}
for any $\gamma \in B$, with a similar combinatorial meaning. For proving the formula for $y_{I}$ using products of partial derivatives of $f$, one can establish an analogue of Corollary \ref{formwden} stating that $y_{I}=\sum_{\gamma=\{s_{H,t}\}_{H,t} \in B_{I}}\Big[d_{\gamma}\prod_{H,t}f_{Hy^{t}}^{s_{H,t}}\Big/f_{y}^{\sum_{H,t}s_{H,t}}\Big]$, and the main result is that for $\gamma \in B_{I,g}$ the coefficient $d_{\gamma}$ equals $(-1)^{g}D_{\gamma}$ (note that this covers the case with $I=\{i\}$, hence $|I|=1$, where we already know that $y_{i}=-\frac{f_{i}}{f_{y}}$, since the unique element $\gamma$ of $B_{i}=B_{i,1}$ satisfies $D_{\gamma}=1$ via Equation \eqref{Dgamma}, and hence $d_{\gamma}=-1$). We express this result as follows.
\begin{thm}
For every multiset $I$, of size $n\geq1$, we have the equality
\[y_{I}=\sum_{g=1}^{2n-1}\sum_{\gamma=\{s_{H,t}\}_{H,t} \in B_{I,g}}\frac{(-1)^{g}D_{\gamma}}{f_{y}^{g}}\prod_{H,t}f_{Hy^{t}}^{s_{H,t}},\] where the coefficient $D_{\gamma}$ is the one from Equation \eqref{Dgamma}. \label{claselts}
\end{thm}
The expression from Theorem \ref{claselts} has the same required symmetry property that was mentioned after Theorem \ref{coeffval}, namely invariance under permutations that preserve $I$.

Now, Theorem \ref{claselts} can be proved by induction, using an analogue of Proposition \ref{indcomb} that is suitable for the action of differentiation on such expressions (such a direct proof of Theorem 17 of \cite{[Z]}, i.e., in the case of ordinary derivatives, is detailed in \cite{[Wi]}), but we will prove it, for $|I|\geq2$, using Theorem \ref{coeffval} (and we already mentioned that the case where $|I|=1$ is the known, classical formula). Considering an element $\gamma=\{s_{H,t}\}_{H,t} \in B$, let $Z_{\gamma}$ be the set of all collections of numbers $\{q_{H,t,K}\}_{|H|+t\geq2,|K| \leq t}$ satisfying $\sum_{|K| \leq t}q_{H,t,K}=s_{H,t}$ for every $H$ and $t$ with $|H|+t\geq2$, as well as $\sum_{|H|+t\geq2}\sum_{|K| \leq t}\kappa_{i}q_{H,t,K}=s_{i,0}$ for every $1 \leq i \leq N$, where $\kappa_{i}$ is the multiplicity of $i$ in the multiset $K$. We complement the notation $J!$ and $\binom{J}{K}$ for multisets by writing $\binom{t}{K}$ for the multinomial coefficient $\binom{t}{\kappa_{1},\ldots,\kappa_{N}}=\frac{t!}{(t-|K|)!\prod_{i=1}^{N}\kappa_{i}!}=\frac{t!}{K!(t-|K|)!}$ (which vanishes unless $t\geq|K|$), and prove the following lemma.
\begin{lem}
The formula for $y_{I}$ in Theorem \ref{coeffval} expands, in terms of the derivatives $f_{Hy^{t}}$ and the explicit expression $\{s_{H,t}\}_{H,t}$ for elements $\gamma \in B_{I,g}$, to yield the sum over $1 \leq g\leq2n-1$ of the sum over $\gamma \in B_{I,g}$ of  \[\Bigg[\!\frac{(-1)^{g}\big(g\!-\!\sum_{i=1}^{N}s_{i,0}\!-\!1\!\big)!I!}{\prod_{H,t}(t!H!)^{s_{H,t}}}\times\!\sum_{\{q_{H,t,K}\}_{H,t,K} \in Z_{\gamma}}\prod_{|K| \leq t}\!\frac{1}{q_{H,t,K}!}\!\binom{t}{K}^{q_{H,t,K}}\!\Bigg]\!\frac{\prod_{H,t}f_{Hy^{t}}^{s_{H,t}}}{f_{y}^{g}}.\] \label{expDeltaJfyr}
\end{lem}

\begin{proof}
Write each expression $\Delta_{l}f_{y^{r}}$ from Theorem \ref{coeffval} via Definition \ref{DeltaJg}. Then the Multinomial Theorem expresses it $m_{J,r}$th power as the sum over all collections of numbers $\{\tilde{q}_{J,r,K}\}_{K \subseteq J}$ satisfying $\sum_{K \subseteq J}\tilde{q}_{J,r,K}=m_{J,r}$ of the product $\prod_{K \subseteq J}\big[(-1)^{|K|}\binom{J}{K}g_{(J \setminus K)y^{|K|}}\cdot\prod_{i \in K}f_{i} \cdot f_{y}^{|J|-|K|}]^{\tilde{q}_{J,r,K}}$ times the multinomial coefficient $m_{J,r}!\big/\prod_{K \subseteq J}\tilde{q}_{J,r,K}!$. We take the exponents $\{m_{J,r}\}_{\{|J|+r\geq2\}}$ to be an element $\alpha$ from the set $A_{I,h}$ from Remark \ref{setpart}, multiply our expressions over $J$ and $r$, and multiply further by the coefficient $c_{\alpha}=(-1)^{h}C_{\alpha}$ from Theorem \ref{coeffval} and Definition \ref{combcoeff}. The factors $m_{J,r}!$ cancel, and when we expand each $\binom{J}{K}$ as $\prod_{i=1}^{N}\binom{\eta_{i}}{\kappa_{i}}$ (using the respective multiplicities of $i$), then the powers of $\eta_{i}!$ forming $J!$ also cancel for every $1 \leq i \leq N$ (but the products of $\kappa_{i}!$ and of $(\eta_{i}-\kappa_{i})!$ merge to $K!$ and $(J \setminus K)!$ respectively). As the product of the expressions $f_{y}^{|J|}$ cancels with $f_{y}^{n}$ from the denominator, and the power $m_{J,r}$ of $r!$ in the denominator of $C_{\alpha}$ (or $c_{\alpha}$) is $\sum_{K \subseteq J}\tilde{q}_{J,r,K}$, we deduce that in the term corresponding to $\alpha$, the summand arising from the collection $\{\tilde{q}_{J,r,K}\}_{K \subseteq J}$ is
\begin{equation}
\frac{(h-1)!I!\prod_{i=1}^{N}f_{i}^{\sum_{J,r,K}\kappa_{i}\tilde{q}_{J,r,K}}}{(-1)^{h+\sum_{J,r,K}|K|\tilde{q}_{J,r,K}}f_{y}^{h+\sum_{J,r,K}|K|\tilde{q}_{J,r,K}}}\prod_{|J|+r\geq2}\prod_{K \subseteq J}\frac{f_{(J \setminus K)y^{r+|K|}}^{\tilde{q}_{J,r,K}}/\tilde{q}_{J,r,K}!}{\big(r!K!(J \setminus K)!\big)^{\tilde{q}_{J,r,K}}}, \label{powmJrexp}
\end{equation}
where each triple sum $\sum_{J,r,K}$ stands for $\sum_{|J|+r\geq2}\sum_{K \subseteq J}$. Now, Equation \eqref{powmJrexp} does not contain the parameters $m_{J,r}$ of $\alpha$, meaning that when we take the sum over $\alpha \in A_{I,h}$ we obtain the sum of the expression from that equation over all collections of numbers $\{\tilde{q}_{J,r,K}\}_{|J|+r\geq2,K \subseteq J}$ that satisfy the equalities $\sum_{J,r,K}\tilde{q}_{J,r,K}=h$, $\sum_{J,r,K}r\tilde{q}_{J,r,K}=h-1$, and $\sum_{J,r,K}\eta_{i}\tilde{q}_{J,r,K}=\nu_{i}$ for each $1 \leq i \leq N$ (i.e., $\sum_{J,r,K}\tilde{q}_{J,r,K}J=I$).

Now, for each $\{\tilde{q}_{J,r,K}\}_{|J|+r\geq2,K \subseteq J}$, the exponent of $f_{i}$ in Equation \eqref{powmJrexp} is $\sum_{J,r,K}\kappa_{i}\tilde{q}_{J,r,K}$, and we therefore separate our set of collections according to the value $s_{i,0}$ of this sum, for each $1 \leq i \leq N$. This (non-negative) value is bounded by $\sum_{J,r,K}\eta_{i}\tilde{q}_{J,r,K}=\nu_{i}$, and note that the exponent showing up twice in the denominator is just $\sum_{i=1}^{N}s_{i,0}$, which thus lies between 0 and $n$. We recall as well that for the total expression from Theorem \ref{coeffval} we also need to sum over $1 \leq h \leq n-1$. We shall write $H$ for the multiset difference $J \setminus K$ and $t$ for $r+|K|$, so that $|H|+t=|J|+r\geq2$ and $K$ satisfies $|K| \leq t$, and we denote each number $\tilde{q}_{J,r,K}=\tilde{q}_{H+K,t-|K|,K}$ as just $q_{H,t,K}$. As summing over $J$, $r$, and $K$ is the same as summing over $H$, $t$, and $K$ (and the total sum $\sum_{H,t,K}$ means $\sum_{|H|+t\geq2}\sum_{|K| \leq t}$), the restrictions on the numbers in the collection in this notation become $\sum_{H,t,K}q_{H,t,K}=h$, $\sum_{H,t,K}\rho_{i}\tilde{q}_{H,t,K}=\nu_{i}-s_{i,0}$ (where $\rho_{i}$ is the multiplicity of $i$ in $H$), and $\sum_{H,t,K}tq_{H,t,K}=h+\sum_{i=1}^{N}s_{i,0}-1$. We also replace $h$ by $g:=h+\sum_{i=1}^{N}s_{i,0}$, satisfying $1 \leq g\leq2n-1$, multiply and divide by $t!$, and gather the corresponding factorials into the binomial coefficient $\binom{t}{K}$, and then Equation \eqref{powmJrexp}, written in terms of a collection $\{q_{H,t,K}\}_{H,t,K}$ satisfying the restrictions given by the above equalities, becomes
\[\frac{\big(g\!-\!\sum_{i=1}^{N}s_{i,0}\!-\!1\!\big)!I!\prod_{i=1}^{N}f_{i}^{s_{i,0}}}{(-1)^{g}f_{y}^{g}}\!\!\!\prod_{|H|+t\geq2}\!\frac{f_{Hy^{t}}^{\sum_{|K| \leq t}q_{H,t,K}}}{\big(t!H!\big)^{\sum_{|K| \leq t}q_{H,t,K}}}\!\prod_{|K| \leq t}\!\frac{1}{q_{H,t,K}!}\!\binom{t}{K}^{q_{H,t,K}}\!.\] We thus set $s_{H,t}:=\sum_{|K| \leq t}q_{H,t,K}$ for every $H$ and $t$ with $|H|+t\geq2$, complete it with the values for $s_{i,0}$, $1 \leq i \leq N$ from before, and with $s_{\emptyset,0}=s_{\emptyset,1}=0$ we obtain that $\{s_{H,t}\}_{H,t}$ lies in $B_{I,g}$. Noting that $t!H!=1$ where $H$ is a singleton and $t=0$, and that the collections contributing to the expression associated with an element $\gamma \in B_{I,g}$ are precisely those lying in $Z_{\gamma}$, this proves the lemma.
\end{proof}

Our proof of Theorem \ref{claselts} uses the following result.
\begin{prop}
Assume that for every multiset $H$ and integer $t\geq0$ with $|H|+t\geq2$ we have a non-negative integer $s_{H,t}$ such that $s_{H,t}=0$ for all but finitely many pairs $(H,t)$. Define $g$ by the equality $\sum_{|H|+t\geq2}ts_{H,t}=g-1$, and assume that for every $1 \leq i \leq N$ we have an integer $s_{i,0}\geq0$ such that $\sum_{i=1}^{N}s_{i,0} \leq g-1$. Then, if $\gamma$ denotes the full set $\{s_{H,t}\}_{H,t}$ and $Z_{\gamma}$ is  as above, then the sum $\sum_{\{q_{H,t,K}\}_{H,t,K} \in Z_{\gamma}}\prod_{|H|+t\geq2}s_{H,t}!\prod_{|K| \leq t}\binom{t}{K}^{q_{H,t,K}}\big/q_{H,t,K}!$ gives the multinomial coefficient $\binom{g-1}{s_{1,0},\ldots,s_{N,0}}$. \label{sumZgamma}
\end{prop}

We shall give a combinatorial proof to Proposition \ref{sumZgamma}, since an algebraic one seems hard to establish.
\begin{proof}
Assume that we are given $g-1$ balls, and that for every $H$ and $t$ with $|H|+t\geq2$ there are $ts_{H,t}$ balls that are marked with $H$ and $t$ (the total sum $\sum_{H,t}ts_{H,t}$ of the numbers of balls indeed equals $g-1$). Assume further that the balls with the markings $H$ and $t$ arrive packed in $s_{H,t}$ identical boxes, with each box containing $t$ balls. If we wish to collect $N$ disjoint sets of balls, of respective sizes $s_{i,0}$, $1 \leq i \leq N$, from these $g-1$ balls (with a remainder of $g-\sum_{i=1}^{N}s_{i,0}$ balls), then it is known that the number of ways to do so is $\binom{g-1}{s_{1,0},\ldots,s_{N,0}}$. We call this a \emph{collection of sets of balls}, and the sizes $\{s_{i,0}\}_{i=1}^{N}$ of the sets are fixed parameters.

We now count this number of choices in a different way, which is based on the markings of the balls and their partitions into the boxes. Considering a collection of sets of balls, a multiset $H$, and some $t$ with $|H|+t\geq2$, take $K$ to be a multiset with multiplicities $\{\kappa_{i}\}_{i=1}^{N}$, and let $q_{H,t,K}$ denote the set of boxes of balls with markings $H$ and $t$ from which we took precisely $\kappa_{i}$ balls for the $i$th set. The fact that the sets are disjoint and there are $t$ balls in every such box means that $q_{H,t,K}$ can be positive only if $|K| \leq t$. Given such a collection of parameters $\{q_{H,t,K}\}_{H,t,K}$, the number of possibilities to choose which $q_{H,t,K}$ of the $s_{H,t}$ boxes are those from which we take balls according to $K$ is $s_{H,t}!\big/\prod_{|K| \leq t}q_{H,t,K}!$, and after determining such a choice, the number of ways to take $\kappa_{i}$ balls to the $i$th set for every $1 \leq i \leq N$ from the $t$ balls in each of the $q_{H,t,K}$ boxes determined for this is $\binom{t}{K}$. The choices over the boxes are independent, yielding $\prod_{|K| \leq t}\binom{t}{K}^{q_{H,t,K}}$, which we then multiply by previous factor, and as the choices for each $H$ and $t$ are also independent, this gives the product associated with $\{q_{H,t,K}\}_{H,t,K}$. Now, since from every box with $H$ and $t$ we took some (possibly empty) set of balls, the sum $\sum_{|K| \leq t}q_{H,t,K}$ equals the number of such boxes, namely $s_{H,t}$, and for every $1 \leq i \leq N$ the sum $\sum_{H,t,K}\kappa_{i}q_{H,t,K}$ equals the number of balls in the $i$th set, which is $s_{i,0}$. This means that $\{q_{H,t,K}\}_{H,t,K}$ lies in $Z_{\gamma}$, and since every element of $Z_{\gamma}$ contributes in this way, and the contributions count distinct ways to choose the balls, the sum over $Z_{\gamma}$ of our expressions indeed yields the multinomial coefficient from the direct answer to this question. This proves the proposition.
\end{proof}

Theorem 17 of \cite{[Z]} was already known from \cite{[Wi]} (and in some sense \cite{[J3]}). However, this does not seem to be the case for our Theorem \ref{claselts}, in this generality. We therefore supply its proof.
\begin{proof}[Proof of Theorem \ref{claselts}]
Lemma \ref{expDeltaJfyr} already expands the expression from Theorem \ref{coeffval} as a linear combination of the products of partial derivatives of $f$ that are associated with elements of $B_{I}$, with the coefficient $\frac{(-1)^{g}}{f_{y}^{g}}$ already showing up when the element lies in $B_{I,g}$, as does the part $\frac{I!}{\prod_{H,t}(t!H!)^{s_{H,t}}}$ of $D_{\gamma}$ from Equation \eqref{Dgamma}. Proposition \ref{sumZgamma} allows us to replace the sum over $Z_{\gamma}$ in that lemma by $\binom{g-1}{s_{1,0},\ldots,s_{N,0}}\big/\prod_{|H|+t\geq2}s_{H,t}!$. Expanding this multinomial coefficient, the factor $\big(g-\sum_{i=1}^{N}s_{i,0}-1\big)!$ cancels, and we get the remaining numerator $(g-1)!$ as well as the missing factorials $\prod_{i=1}^{N}s_{i,0}!$ in the denominator. Recalling that if $|H|+t\leq1$ then either $H$ is a singleton and $t=0$ or $H$ is empty and $t\leq1$, and that in elements of $B_{I}$ the multiplicities $s_{\emptyset,0}$ and $s_{\emptyset,1}$ must vanish, this is indeed the desired expression. This proves the theorem.
\end{proof}

Note that in every product of derivatives that shows up in Theorem \ref{claselts}, the total differentiation with respect to the variables $x_{i}$, $1 \leq i \leq N$ is given by the multiset $I$, and the number of differentiations with respect to $y$ in the numerator is one less that the power of $f_{y}$ in the denominator. This generalizes the observation from \cite{[Z]} that is based on the prediction from \cite{[N]}. The fact that every term in the expression for $\Delta_{J}f_{y^{r}}$ in Definition \ref{DeltaJg} contains differentiation with respect to the $x_{i}$'s that are given by $J$, and the number of differentiations with respect to $y$ in every such term is $r+|J|$, means that the same observation can be drawn also from Proposition \ref{formnoden} or Corollary \ref{formwden}.

\medskip

We conclude with the following remark about our coefficients $C_{\alpha}$ and $D_{\gamma}$ from Definition \ref{combcoeff} and Equation \eqref{Dgamma}.
\begin{rmk}
Recall from Proposition 1 of \cite{[H]} that the formula for higher partial derivatives of a product of functions, or of a composition of functions, contain no numerical coefficients when the variables are all distinct. Moreover, Proposition 2 there states that in more general derivatives we obtain coefficients only due to formerly distinguishable terms becoming indistinguishable, and the coefficients count the appropriate number of collapsing partitions (the case from Fa\`{a} di Bruno's formula, of an ordinary derivative, is, of course, a special case of this construction). Proposition 4 of that reference states that if a multiset $I$ has a multiset partition as $\sum_{K\neq\emptyset}\mu_{K}K$, then the number of partitions of a set of size $|I|$ that under the identifications that create $I$ collapse to the partition $\sum_{K\neq\emptyset}\mu_{K}K$ is $I!\big/\prod_{K}\mu_{K}!K!^{\mu_{K}}$. A similar argument shows that if we have a marking $v$ of the multiplicities $\mu_{K}$, namely $\mu_{K}=\sum_{v}\mu_{K,v}$ and in collapsing the partitions we remember which part belongs to which marking, then the number of partitions that collapse in the same identifications process is now $I!\big/\prod_{K,v}\mu_{K,v}!K!^{\mu_{K,v}}$ (this is the previous expression multiplies by the product over $K$ of the multinomial coefficients $\mu_{K}!\big/\prod_{v}\mu_{K,v}!$). Now, if the multiset $I$ is a set of size $|I|$, then $I!=1$, the multiplicities $m_{J,r}$ appearing in Definition \ref{combcoeff} or $s_{H,t}$ from Equation \eqref{Dgamma} equal 0 wherever the multiset $J$ or $H$ is not a set and equal 0 or 1 when it is a non-empty set, meaning that $m_{J,r}!J!^{m_{J,t}}$ or $s_{H,t}!H!^{s_{H,t}}$ equal 1 as well unless $J$ or $H$ are empty. In this case the coefficient $C_{\alpha}$ for $\alpha \in A_{I,h}$ reduces to $(h-1)!\big/\prod_{r\geq2}m_{\emptyset,r}!\prod_{|J|+r\geq2}r!^{m_{J,r}}$, and for an element $\gamma \in B_{I,g}$, the coefficient $D_{\gamma}$ becomes $(g-1)!\big/\prod_{t\geq2}s_{\emptyset,t}!\prod_{H,t}t!^{s_{H,t}}$ (the products on $r$ or $t$ alone begin with 2 because of the restriction $|J|+r\geq2$ for the $m_{J,r}$'s and $s_{\emptyset,0}=s_{\emptyset,1}=0$ for the $s_{H,t}$'s). These can be viewed as the \emph{fundamental parts} of the coefficients (unlike the case corresponding to Fa\`{a} di Bruno's formula, they do not always equal 1---see Equation \eqref{yder4} for the first case where the former equals 2). Our generalization of Proposition 4 of \cite{[H]} mentioned here shows that in general, the coefficient $C_{\alpha}$ or $D_{\gamma}$ is the product of its fundamental part and the corresponding number of partitions that collapse to the desired one when a set of cardinality $|I|$ gets identified to produce the multiset $I$. However, the formula for general coefficients is not more complicated than the formula for their fundamental parts, which is why it is more natural (unlike in \cite{[H]}) to consider general differentiation from the start, rather than begin with differentiation with respect to distinct variables. \label{Hardy}
\end{rmk}

\noindent\textsc{Einstein Institute of Mathematics, the Hebrew University of Jerusalem, Edmund Safra Campus, Jerusalem 91904, Israel}

\noindent E-mail address: zemels@math.huji.ac.il

\end{document}